\documentclass{amsart}
\usepackage{epsfig}
\theoremstyle{plain}
\newtheorem{theorem}{Theorem}[section]
\newtheorem{lemma}[theorem]{Lemma}
\newtheorem{proposition}[theorem]{Proposition}
\newtheorem{corollary}[theorem]{Corollary}

\theoremstyle{definition}

\theoremstyle{remark}
\newtheorem*{remark*}{Remark}
\newtheorem*{claim}{Claim}

 \def\bd{\partial}
 \def\bdM{{\partial \N}}
 \def\C{{\mathbb C}}
 \def\R{{\mathbb R}}
 \def\Q{{\mathbb Q}}
 \def\E{{\mathbb E}}
 \def\H{{\mathbb H}}
 \def\Z{{\mathbb Z}}
 \def\g{{\mathfrak g}}
 \def\lap{\triangle}

\def\curl{{\rm curl}}
\def\tr{\,{\rm tr}}

\def\arcsinh{\mathop{\rm arcsinh}}

\def\arctanh{{\mathop{\rm arctanh}}}
\def\area{\mathop{\rm area}}
\def\grad{{\nabla}}
\def\vol{\mathop{\rm vol}}

\def\M{{X}} 
\def\N{{M}} 
\def\b{{b}} 
\def\L{L} 
\def\z{\zeta} 
 
\def\DS{{\cal H} {\cal D} {\cal S}}

\def\W{{Weitzenb\"ock  }}

\def\cal{\mathcal}
\def\Re{\mathop{\rm Re}}
\def\Im{\mathop{\rm Im}}

\def\r{\rho}

\def\n{{n}}
\def\curl{{\rm curl}\,}
\def\rot{{\rm rot}\,}
\def\div{\mathop{\rm div}}
\def\tr{\mathop{\rm tr}}
\def\grad{\mathop{\rm grad}}
\def\la{{\langle}}
\def\ra{{\rangle}}
\def\s{\sigma}

\def\cl{{\cal L}}
\def\ca{{\cal A}}

\title{The shape of  hyperbolic Dehn surgery space}
\author{Craig D. Hodgson}
\address{Department of Mathematics and Statistics\\
University of Melbourne\\
Victoria 3010, Australia}
\email{cdh@ms.unimelb.edu.au}
\thanks{The research of the first author was partially
supported by grants from the ARC}
\author{Steven P. Kerckhoff}
\address{Department of Mathematics\\
Stanford University\\
Stanford, CA 94305, U.S.A.}
\email{spk@math.stanford.edu}
\thanks{The research of the second author was partially
supported by grants from the NSF}

\dedicatory{Dedicated to Bill Thurston on his 60th birthday.}

\begin{document}

\begin{abstract} 
In this paper we develop a new theory of infinitesimal harmonic deformations for 
compact hyperbolic 3-manifolds with ``tubular boundary''. In particular, this applies to 
complements of tubes of radius at least $R_0 = \arctanh(1/\sqrt{3}) \approx 0.65848$ 
around the singular set of hyperbolic cone manifolds, removing the previous 
restrictions on cone angles.  

We then apply this to obtain a new quantitative version of Thurston's hyperbolic Dehn 
surgery theorem, showing that all generalized Dehn surgery coefficients outside a disc of 
``uniform'' size yield hyperbolic structures. Here the size of a surgery coefficient is 
measured using the Euclidean metric on a horospherical cross section to a cusp in the 
complete hyperbolic metric, rescaled to have area 1.  We also obtain good estimates on the 
change in geometry (e.g. volumes and core geodesic lengths) during hyperbolic Dehn filling.

This new harmonic deformation theory has also been used by Bromberg and his coworkers in 
their proofs of the Bers Density Conjecture for Kleinian groups.
\end{abstract}

\maketitle

\section{Introduction}\label{intro}

Let $\M$ be a compact, orientable $3$-manifold with a finite number of
torus boundary components and suppose its interior admits a complete, finite volume
hyperbolic metric.  For each torus, there are an
infinite number of topologically distinct ways to attach a solid torus, corresponding
to the homotopy class of the non-trivial simple closed curve on the boundary torus
that bounds a disk in the solid torus.   The set of non-trivial simple closed curves
on a torus is parametrized by pairs of relatively prime
integers, once a basis for the fundamental group of the torus
is chosen.  If each torus is filled, the resulting manifold is
closed.  A fundamental theorem of Thurston (\cite{thnotes})
states that, for all but a finite number of filling curves on
each boundary component, the resulting closed $3$-manifold has a
hyperbolic structure.  Those that don't result in a hyperbolic structure are 
called {\em exceptional} curves.  In \cite{HK2} we showed that there is a universal
bound to the number of curves that must be excluded from each boundary
component; in particular, when there is one boundary component, there are at most 
$60$ exceptional curves and when there are multiple boundary components, at most 
$114$ curves from each component need to be excluded.

The proof of this result uses the harmonic deformation theory developed in \cite{HK1}  
to deform the finite volume complete hyperbolic structure on the interior of $\M$.  
The deformation consists of a 
family of singular hyperbolic metrics on the filled manifolds where the 
singularities lie along geodesics isotopic to the cores of the attached solid tori.  The metrics
on  discs perpendicular to the geodesics have a single cone point at the geodesic 
with a cone angle $\alpha$ which is constant along each component.   We call such
structures {\it hyperbolic cone manifolds}.  At the complete structure the cone angles
are considered to be equal to $0$ and if the deformation reaches cone angles
$2 \pi$ on each component, this is a smooth hyperbolic structure on the filled
manifold.

For the analytic techniques in \cite{HK1} to work, it is necessary to restrict all the cone
angles to be at most $2 \pi$.  While this is adequate for the Dehn filling problem,  there
are other situations where it is important to be able to deal with larger cone angles or
with a more general type of singularity.  In particular, in his proof of the Bers
Density Conjecture, Bromberg (\cite{Brm1}, \cite{Brm2}) (and then Brock-Bromberg (\cite{BB}) 
for more general versions of 
the Density Conjecture) needs to deform cone angles equal to $4 \pi$ back to $2 \pi$. In order 
to have a similar deformation theory in this type of situation, a new analytic technique 
is necessary.  One of the primary goals of this paper is to provide such a technique.  
To this end, the previous analysis, which developed an $L^2$ Hodge theory for 
the incomplete smooth metric on the complement of the singular locus, is replaced 
by a Hodge theory on the compact hyperbolic manifold with boundary 
obtained by removing an open tubular neighborhood of the singular locus.  

\bigskip
We will now give a brief outline of this theory and some of its applications, as they
are presented in the body of this paper.
\bigskip

If one removes an equidistant tubular neighborhood of each component of the singular
locus of a $3$-dimensional cone manifold, one obtains a smooth hyperbolic manifold 
with torus boundary components.
The boundary tori have intrinsic flat metrics.  Furthermore, the principal normal curvatures
are constant on each component, equal to $\kappa, \frac{1} {\kappa}$
(where we assume that $\kappa \geq 1$).  If $\kappa \neq 1$ the lines of curvature are
geodesics in the flat metric corresponding to the meridional and
longitudinal directions.  The normal curvatures and the tube radius, $R$, are related 
by ${\rm coth}~R = \kappa$ so they determine each other.  We call tori with
these curvature properties {\it tubular}. 

We say that an orientable hyperbolic $3$-manifold $\N$ has {\it tubular boundary} if its 
boundary components are flat tori with constant principal normal curvatures
as above.  The interior of $\N$ is 
known to also admit a {\it complete} finite volume hyperbolic metric 
(\cite[Lemma 3.8]{HK2}).
Associated to $\N$ is a {\it filled} hyperbolic manifold $\hat M$.
It is canonically obtained from $\N$ by extending the hyperbolic structures on the
boundary tori normally as far as possible.  Each added component is topologically
a torus crossed with $\R$ and is foliated by tubular tori whose radii go to zero.  
Then $\hat M$ is homeomorphic to the interior of $\N$ and has a (typically incomplete)
hyperbolic structure.  The original manifold with tubular boundary $\N$ is 
a subset of $\hat M$ obtained by truncating the ends of $\hat M$ along the 
appropriate tubular tori.  Note that $\hat M$ contains many other manifolds
with tubular boundary obtained by choosing other truncations.

The geometry of the ends of the filled manifold $\hat M$ is completely determined
by the geometry of the tubular boundary tori of $\N$.  In particular, if, for
a particular boundary component, the lines of curvature for the larger curvature 
$\kappa >1$ are parallel closed curves, the completion of that end of $\hat M$ will have 
the structure of a hyperbolic cone manifold with those curves as meridians around 
the singular locus.  The resulting cone angles can be read
off from the geometry of the boundary and can be arbitrarily large.  In general,
these lines of curvature merely determine a foliation on the boundary torus
where the leaves are geodesic in the flat metric on the torus.  There is 
still a canonical way to extend the structure of the boundary torus in this case 
but the singular set for the completion will be a single point with a complicated neighborhood.  
The resulting structure obtained by completing $\hat M$, including both cone manifolds and 
this more general type of singularity, is called a hyperbolic structure with 
{\em Dehn surgery singularities}.  (See \cite{thnotes} for details.)

Fix a component $T$ of the boundary of a hyperbolic $3$-manifold $\N$ 
with tubular boundary
and consider the holonomy group of the fundamental group of that boundary torus.
Assume that $\kappa > 1$; then each element in the holonomy group will have an 
invariant axis.  Since the group is abelian, all of the elements fix a common axis 
in $\H^3$.  Choose a direction along the axis.  Then, associated to each element 
is a {\it complex length} whose real part is the signed translation distance along 
the axis and whose imaginary part is the total rotation around the axis.   The amount 
of rotation is a well-defined real number whose sign is determined by the right hand rule. The map sending an element of $\pi_1 (T) = H_1(T;\Z)$ to its complex length is linear
and can be extended canonically to a linear map  $\cl: H_1(T;\R) \to \C$.  The
resulting value $\cl (c)$ for any element $c \in H_1(T;\R)$ will be called the complex length of $c$.

When the complex length of a {\it simple} closed curve $\gamma$ on $T$ equals
$2 \pi i$, this implies that the end of the filled manifold $\hat M$ corresponding 
to $T$ completes to a smooth structure on the manifold $\N(\gamma)$
obtained from $\N$ by Dehn filling with filling curve $\gamma$.  More generally,
if the complex length of $\gamma$ equals $\alpha i$, the end completes to a hyperbolic
cone structure on $\N(\gamma)$ with cone angle $\alpha$.  Now assume that 
the map $\cl: H_1(T;\R) \to \C$ is a (real) isomorphism.  (This holds whenever the
complex lengths of two generators of $H_1(T;\Z)$ are linearly independent over the
reals, a condition that holds whenever $\N$ has tubular boundary with $\kappa >1$ on $T$.)
Then there will be a
unique element $c \in H_1(T;\R)$ so that $\cl (c) = 2 \pi i$; we say that $c$ 
is the {\it Dehn surgery coefficient} of the boundary component $T$.  When $c \in 
H_1(T;\Z)$ is a primitive element, it corresponds to a simple closed curve and
the completion of the end is smooth.  When it is in $H_1(T;\Q)$, the completed end
has a cone singularity along a core geodesic.

Suppose, for simplicity, that $\N$ has a single boundary component $T$ and denote
by $\M$ the underlying smooth manifold with boundary. Then the subset 
of $H_1(T;\R)$ consisting of Dehn surgery coefficients of hyperbolic structures
with tubular boundary on $\M$ is called
the hyperbolic Dehn surgery space for $\M$ and will be denoted by $\DS(\M)$.
(Note: In \cite{thnotes} hyperbolic Dehn surgery space is equivalently defined in terms of the filled in structures.)
Thurston's theorem about the finiteness of exceptional curves is actually a corollary 
of his theorem that $\DS(\M)$  contains a neighborhood of infinity (infinity here
corresponds to the complete finite volume structure on the interior of $\M$).
Hence, it contains all but a finite number of points of the integral lattice 
$H_1(T;\Z) \subset H_1(T;\R)$.
Similar statements are proved when there are multiple boundary components.

Thurston's proof is not effective; it gives no information about the size or shape of
hyperbolic Dehn surgery space which is why there is no information
from his proof about the number of exceptional curves in his finiteness theorem.
Note that  the vast majority of the points with integral entries are ``near" infinity
and thus that the statement that $\DS (\M)$ contains a neighborhood of infinity is 
really a much stronger statement than the finiteness of exceptional fillings.

One of the main goals of this paper is to provide an effective proof of Thurston's
result, one that will guarantee that $\DS (\M)$ contains a neighborhood of infinity
of a ``uniform" size and shape.  To simplify the description of this uniform 
region, it is useful to put a metric on $H_1(T;\R)$. 
One way to do this is to consider the complete structure on the interior of $\M$ 
and take a horospherical torus $T$ embedded in its end.  This torus inherits a
flat metric which is well-defined up to scale. The homology group $H_1(T;\R)$
can be canonically identified with the universal cover of $T$.   
So the flat metric on $T$, normalized to have area $1$, induces a flat metric 
on $H_1(T;\R) \cong \R^2$.  Note that under this identification, the distance
from the origin to point $c \in H_1(T;\Z)$ is just the geodesic length of the
corresponding closed curve $\gamma$ on $T$ measured with respect to the 
flat metric on $T$, normalized to have area $1$.  This is called the
{\it normalized length} of $\gamma$ on $T$.  As with complex length
this notion of length can be extended naturally to define a map
$\hat L: H_1(T;\R) \to \R$; the value $\hat L(c), c \in H_1(T;\R)$ is called the
normalized length of $c$.

The theorem below says that, using this metric on the plane, $\DS(\M)$ always 
contains the complement of a disk of uniform radius around the origin, {\em independent} of $\M$.

\begin{theorem} \label{shapethm}
Consider a complete, finite volume hyperbolic structure on the interior of a compact, orientable $3$-manifold $\M$ with one torus boundary component.  Let $T$ be a horospherical torus which is embedded as a cross-section to the cusp of the complete structure.   Consider $\DS(\M)$ as a subset of $H_1(T;\R) \cong \R^2$ where 
the latter is endowed with the Euclidean metric induced from the universal cover of $T$ with its flat metric scaled to have unit area.
Then $\DS(\M)$ contains the complement of a disk of radius  $7.5832$, centered
at the origin. Equivalently, any $c \in H_1(T;\R)$ whose normalized length
$\hat L (c)$ is bigger than $7.5832$ is in $\DS(\M)$.
\end{theorem}

In \cite {HK2} we showed that any simple closed curve $\gamma$ on $T$, viewed as an element of
$H_1(T;\Z)$, whose normalized length is at least $7.515$ is in $\DS(\N)$.  Thus, except for the slight change
in constant (which is due to the tube radius condition required for the
Hodge theorem, as discussed below), Theorem \ref{shapethm} is a direct generalization
of that result. The normalized length condition translates easily into an upper
bound on the number of exceptional fillings.
  
The proof of Theorem \ref{shapethm}, 
like the proof of the uniform bound on exceptional fillings,
involves two main steps.  First, it is necessary to show that one can deform a
given structure towards the desired structure.  For example, in the cone
manifold case, one needs to show that the cone angles can always be increased
a small amount.  This step depends on proving a local parametrization 
theorem, showing that, locally, the deformations are parametrized by their 
Dehn surgery coefficients.  In order to find such a local parametrization, one 
proves a local rigidity theorem which says that it is impossible to deform the
hyperbolic structure while keeping the Dehn surgery coefficients fixed.  The
local parametrization then follows by an application of the implicit function theorem.
 
The second step is to show that, under certain initial conditions, it is 
always possible to deform the complete structure on the interior of $\M$, 
through hyperbolic structures with tubular boundary, to one with the desired 
Dehn surgery coefficient before there is any degeneration of the hyperbolic structure. This step requires one to control the change in geometry under the deformation and depends on the analysis in the proof of local rigidity.

The proofs of the local rigidity and local parametrization theorems require
new analytic techniques and occupy the next three sections.  Once these
are established, the arguments to establish uniform bounds closely follow
those in \cite{HK2}.  However, the use of manifolds with tubular boundary,
as opposed to cone manifolds, leads to subtly different estimates when
there are multiple boundary components.

For a compact, orientable, $3$-manifold $\M$ with multiple torus boundary components $T_1, \cdots, T_k$, the Dehn surgery space is
a subset of $\oplus_i H_1(T_i;\R)$ and the Dehn surgery coefficient 
$c = (c_1, c_2, \cdots, c_k)$ is determined by the Dehn surgery coefficients 
$c_i \in H_1(T_i;\R)$ for each torus.  
For $c_i \in H_1(T_i;\R)$ the normalized length $\hat L(c_i)$  is computed on 
a horospherical torus corresponding to $T_i$ in the complete structure 
on the interior of $\M$ as described
above.  In this case we prove the following uniform statement:

\begin{theorem} \label{multipleshapethm}
Consider a complete, finite volume hyperbolic structure on the interior of a compact, orientable $3$-manifold $\M$ with $k \ge 1$ torus boundary components.  
Let $T_1, \cdots, T_k$ be horospherical tori which are embedded as 
cross-sections to the cusps of the complete structure.   
Consider $\DS(\M)$ as a subset of $\oplus_i H_1(T_i;\R)$. 
Then there exists a universal constant $C = 7.5832$ such that 
$c = (c_1,c_2, \cdots, c_k)$ is in $\DS(\M)$ provided the normalized lengths
$\hat L_i = \hat L(c_i)$ satisfy 
\begin{eqnarray}
\sum_i {1 \over {\hat L_i}^2} < {1 \over {C^2}}. 
\label{hatLsum}
\end{eqnarray}
\end{theorem}

When $k = 1$ this is precisely the same statement as that of Theorem \ref{shapethm}. 
In the multiple cusp case, it again gives
a uniform upper bound on the number of exceptional simple closed curves that
need to be excluded from each boundary component so that the remaining Dehn
filled manifolds are necessarily hyperbolic. 

However, it should be noted that the bound depends on the number of boundary 
components.  This is in contrast to Theorem 5.12 in \cite{HK2} which provides 
a uniform bound independent of the number of cusps.  The reason for this 
difference is that in the previous paper we allow the possibility of
increasing the cone angles at varying rates.  Once a cone angle of $2 \pi$
is attained on one component of the singular locus, it is no longer changed,
while the other angles are increased.  We no longer keep track of the geometry 
in a neighborhood of the smooth core geodesics.  (Indeed, a geodesic could become 
non-simple and change isotopy class.)  This was not adequately explained in
\cite{HK2}; for a discussion of this and other subtler issues that arise
in the multiple cusp case, the reader may consult \cite{Pur}. 

In the current paper, we require a lower bound on the tube radius of all the tubular boundary components throughout the deformation.  This is because, even at the final 
Dehn surgery coefficient, the filled manifold may still have singularities, and thus we
no longer have the luxury of ignoring the tube radius around a component once 
its desired surgery coefficient is attained.  
We always move radially in $\DS(\M)$ from the complete structure to the desired Dehn surgery 
coefficients (which, when all the coefficients
correspond to simple closed curves, amounts to increasing the cone angles at
equal rates).  This provides weaker estimates in the case of multiple cusps.
It does, however, have the advantage in the case of smooth Dehn filling (or
cone manifolds) that the isotopy class of the union of the core geodesics 
will necessarily remain unchanged.
 
We also obtain good control on the change in geometry during generalized 
hyperbolic Dehn filling. Theorem \ref{geomthm} gives explicit upper and lower bounds for the volume and core geodesic length, with the asymptotic behavior given by Neumann-Zagier in \cite{neumann-zagier}. These bounds are illustrated in Figures \ref{fig2}
and \ref{fig3} at the end of the paper. For example, we obtain the following
numerical bound.

\begin{theorem} Let $X$ be a compact, orientable $3$-manifold as in 
Theorem \ref{multipleshapethm},  
and let $V_\infty$ denote the volume of the complete hyperbolic structure
on the interior of $X$.
Let $c =(c_1,\ldots,c_k)\in H_1(\bd X;\R)$
be a surgery coefficient with normalized lengths $\hat L_i = \hat L (c_i)$
satisfying
$$ \sum_i {1 \over \hat L_i^2} < {1 \over C^2} \text{ where } C = 7.5832,$$
and let $M(c)$ be the filled hyperbolic manifold with Dehn surgery coefficient $c$.
 Then the decrease in volume $\Delta V = V_\infty - \vol(M(c))$ during hyperbolic Dehn filling  is at most $0.198$.
\end{theorem}

We now briefly explain how the use of manifolds with tubular boundary
allows us to avoid the analytic issues that led to the cone angle restriction
in \cite{HK2}.

The original local rigidity theory in \cite{HK1} applies only
to hyperbolic $3$-manifolds with conical singularities along a geodesic link where
the cone angles are restricted to be at most $2 \pi$. The argument
involves finding, for any infinitesimal deformation of the hyperbolic cone structure,  
a harmonic representative and then utilizing  a Weitzenb\"ock formula for such harmonic
infinitesimal deformations.  
The analysis using this formula involves an integration by parts on the
complement of a tubular neighborhood of the singular locus, resulting in a term from the
boundary of the tubular neighborhood.   Any deformation for which this boundary term
goes to zero as the radius of the tube goes to zero is seen to be trivial.  The main
step in \cite{HK1} is to show that, for any infinitesimal deformation where all the cone
angles are held constant,  the boundary term does, indeed, go to zero as long as the
cone angles are at most $2 \pi$.  The fact that the analysis involves arbitrarily
small neighborhoods of the singular locus means that it depends on the asymptotic behavior
of harmonic deformations near the singular locus. This behavior is strongly governed by
the value of the cone angle along the singular locus.

To avoid a dependence on the local behavior near the singular locus, it is
necessary to work on the complement of a tubular neighborhood whose tube
radius is bounded below.  Again, one must find a harmonic representative
for the infinitesimal deformation of the hyperbolic structure with tubular boundary.
Since these infinitesimal deformations can be viewed as cohomology classes,
this can be viewed as Hodge theory on a manifold with boundary.  In this case,
when one uses the Weitzenb\"ock formula and integration by parts, one wants
to end up with a boundary term with an appropriate sign.  When such a boundary
term is obtained, the conclusion is again that the deformation is trivial, implying
a local rigidity theorem as before.

Thus, a Hodge theory must be developed with the boundary term from the
Weitzenb\"ock formula in mind.  To this end, we find a formula for this term
in Section \ref{weitz}.  The formula derived there is quite general and is valid 
for any  hyperbolic
3-manifold with boundary, not just those with tubular boundary.  As  a result,
it should have applications in other contexts and may be of independent interest.

The form of this Weitzenb\"ock boundary term motivates the specific boundary conditions
we require for our Hodge representative when the manifold has tubular boundary.   
The proof that the corresponding boundary
value problem can always be uniquely solved is contained in Section \ref{bvalues}.
This result requires a universal lower bound on the tube radius of the
tubular boundary components.   By definition the {\em tube radius} $R$ of a
tubular boundary component is determined by the formula ${\rm coth}\,R = \kappa$,
where $\kappa \geq 1$ is the larger of the two principal curvatures on that component.
When $\kappa = 1$ (which corresponds to a horospherical torus), the tube radius
is said to be infinite.

Once the required Hodge theorem is proved,  similar arguments to those
in \cite{HK1} imply the following local rigidity and local parametrization
result.  The previous cone angle restriction has been removed and is
replaced by a mild restriction on the tube radius.  For simplicity,
we also assume that all tube radii are finite.

\begin{theorem}\label{bdrylocalparam} Let $\N$ be a compact, orientable hyperbolic 3-manifold 
with tubular boundary and suppose that the tube radius of each boundary component 
is finite and at least $R_0 = \arctanh(1/\sqrt{3}) \approx 0.65848.$  Then there are no
deformations of the hyperbolic structure fixing the Dehn surgery coefficient of $\N$.  
Furthermore, the nearby
hyperbolic structures with tubular boundary are parametrized by their Dehn surgery 
coefficients.  In particular,  a finite volume hyperbolic cone-manifold 
with singularities along a link and tube radii at least  $R_0$
has no deformations of the hyperbolic structure keeping the
cone angles fixed, and the nearby hyperbolic cone-manifold structures are
parametrized by their cone angles.
\end{theorem}

Once this analytic theory for hyperbolic manifolds with tubular boundary is
developed and the above local rigidity theorem is proved, the arguments
in \cite{HK2} go through with minor changes.  Indeed, much of that paper
was written in the context of manifolds with tubular boundary, once the necessary
analytic and geometric control was derived.  These arguments are recalled
in Section \ref{applic},  where they are then applied to prove Theorems \ref{shapethm}
and \ref{multipleshapethm} and other results.

\section{Preliminary Material}\label{prelim}

In this section we recall the basic setup for the harmonic deformation theory
of hyperbolic structures on $3$-manifolds.  The reader is referred to
the papers \cite{HK1} and \cite{HK2} for details, and to \cite{HKwarwick} for
a survey of the theory and its applications.

\bigskip

An infinitesimal deformation of a hyperbolic structure on a hyperbolic 
$3$-manifold $\N$ is
given by a cohomology class in $H^1(\N; E)$ where is $E$ is
the bundle of (germs of) infinitesimal isometries of $\H^3$.  
By viewing this cohomology group in terms of de Rham cohomology,
such a cohomology class can be represented by a $1$-form with values
in $E$.  The $1$-form will be closed with respect to the $E$-valued exterior
derivative which we denote by $d_E$.  A representative for a cohomology class can 
be altered by a coboundary
without changing its cohomology class.  An $E$-valued $1$-form is a
coboundary precisely when it can be expressed as $d_E s$, where
$s$ is an $E$-valued $0$-form, i.e. a {\em global} section of $E$.

A standard method for choosing a particularly nice representative in a cohomology
class is to find a {\em harmonic} representative: one that is co-closed
as well as closed.  On a closed manifold such a harmonic representative is
unique.  When the manifold is non-compact or has boundary, it is necessary to
choose asymptotic or boundary conditions to guarantee existence and uniqueness.

When $\N$ is a hyperbolic $3$-manifold with tubular boundary, one can begin with a 
representative $\hat \omega \in H^1(\N; E)$ that has a special form in a  neighborhood 
of the boundary.  In \cite{HK1}, when $\N$ is a hyperbolic cone manifold, specific closed 
$E$-valued $1$-forms, which we call {\em standard forms}, are defined in a neighborhood of
the singular locus.  The same forms are defined in the neighborhood of the boundary 
components of a general hyperbolic manifold with tubular boundary.  
They have the property  that  any possible infinitesimal change 
in the holonomy representation of the fundamental group of a boundary torus can be induced
by one of these forms.  As a result, by standard cohomology theory, for any infinitesimal 
deformation of the hyperbolic cone manifold structure, it is possible to find a closed
$E$-valued $1$-form $\hat \omega$ on $\N$ which equals one of these
standard forms in a neighborhood of each torus boundary.

The standard forms are harmonic so the $E$-valued $1$-form $\hat \omega$
will be harmonic in a neighborhood of the boundary but not generally
harmonic on all of $\N$.  Since it represents a cohomology class in
$H^1(\N;E)$, it will be closed as an $E$-valued $1$-form, but it won't
generally be co-closed.  If we denote by  $\delta_E$ the
adjoint of the exterior derivative, $d_E$, on $E$-valued forms, then this means
that $d_E \hat \omega = 0$, but $\delta_E \hat \omega \neq 0$ in general.
Finding a harmonic (i.e. $d_E$ closed and $\delta_E$ co-closed)
representative cohomologous to $\hat \omega$ is equivalent to finding
a global section $s$ such that
\begin{eqnarray}
\delta_E d_E s ~ = ~ - \delta_E \hat \omega.
\label{correction}
\end{eqnarray}
Then, $\omega = \hat \omega + d_E s$ satisfies
$\delta_E \omega = 0, \, d_E \omega = 0$;  so it is a closed and co-closed
representative in the same cohomology class as $\hat \omega$.

The fibers of the bundle $E$ are all isomorphic to the Lie
algebra $\g$ of the Lie group of isometries of hyperbolic space.
One special feature of the 3-dimensional case is the
{\em complex structure} on the Lie algebra
$\g \cong sl_2 \C$.
The infinitesimal rotations fixing a point $p \in \H^3$
can be identified with $su(2) \cong so(3)$, and
the infinitesimal pure translations at $p$
correspond to $i \, su(2) \cong T_p \H^3$.  
Geometrically, if $t \in T_p \H^3$ represents an infinitesimal translation,
then $i \, t $ represents an infinitesimal rotation with
axis in the direction of $t$. Thus, on a hyperbolic 3-manifold $\N$ we can identify
the bundle $E$ with the {\em complexified} tangent bundle $T\N \otimes \C $.
At each point, the fiber decomposes into a real and imaginary part, representing an
infinitesimal translation and an infinitesimal rotation, respectively, and we can 
speak of the real and imaginary parts of an $E$-valued form.

In \cite{HK1} it was shown that in order to solve equation (\ref{correction})
for $E$-valued sections, it suffices to solve the real part of the equation.
The real part of a section $s$ of $E$ is just a (real) section of the tangent
bundle of $\N$; i.e., it is a vector field, which we denote by $u$.  The
real part of $\delta_E d_E s$ equals $(\nabla^* \nabla + 2) \,u$, where
$\nabla$ denotes the (Riemannian) covariant derivative and $\nabla^*$
is its adjoint.  The composition $\nabla^* \nabla$ is sometimes called
the ``rough Laplacian" or the ``connection Laplacian".  

To solve the real part of the equation (\ref{correction}), we find that the
computations are somewhat easier if we replace vector fields by their
dual real-valued $1$-forms.   We take the real part of $- \delta_E \hat \omega$, 
considered as a vector field, and denote its dual $1$-form by $\zeta$.   The operator 
$(\nabla^* \nabla + 2)$ on vector fields becomes $\hat \lap +4$ on dual $1$-forms, where 
$\hat \lap = \hat d \hat \delta + \hat \delta \hat d$ is the usual Laplacian on real-valued $1$-forms 
and we are denoting the exterior derivative and its adjoint on $\N$ by 
$\hat d$ and $\hat \delta$, respectively.   We must then solve the equation
\begin{eqnarray} (\hat \lap + 4) \, \tau ~=~ \zeta, \label{taueqn}
\end{eqnarray}
for a globally defined real-valued $1$-form $\tau$ on $\N$, which will be dual to the 
vector field $u$.  

For a manifold with boundary it is necessary to prescribe boundary conditions
on $\tau$ for this problem to have a unique solution.  The boundary conditions
we choose are non-standard and very specific to our hyperbolic deformation theory context.
In particular, the local rigidity results that we seek depend on a Weitzenb\"ock
formula for harmonic $E$-valued $1$-forms.  This formula contains a boundary
term whose sign is crucial to the argument. The harmonic form $\omega \in H^1(\N; E)$ 
is obtained by solving (\ref{taueqn}),  which, in turn, 
gives us a solution to (\ref{correction}).  Since $\omega = \hat \omega + d_E s$
and $\hat \omega$ is standard in a neighborhood of the boundary, our boundary term 
will have a contribution from the standard form, which is quite explicit, and from
the correction term $d_E s$.  The behavior of the latter depends on our choice
of boundary condition when solving (\ref{taueqn}).  A major consideration
when choosing a boundary condition is that the contribution to the Weitzenb\"ock
boundary term from the correction term $d_E s$ be non-positive. 

In Section \ref{weitz} we compute a formula for this contribution
for a general hyperbolic $3$-manifold with boundary.  We then specialize
to our current situation of a hyperbolic manifold with tubular boundary and
choose boundary conditions specific to this case.
In Section \ref{bvalues} we prove that the problem of solving (\ref{taueqn})
with these boundary conditions always has a unique solution.

We now recall the Weitzenb\"ock formula for harmonic $E$-valued $1$-forms,
referring to \cite{HK1} and \cite{HK2} for details and proofs.

We can decompose any $\omega \in H^1(\N,E)$  into its real and imaginary
parts $\omega = \eta ~+~ i \, \tilde \eta$, where $\eta$ and $\tilde \eta$ are
vector field valued $1$-forms on $\N$ which we can view as elements of
${\rm Hom}\,(T\N,T\N)$ at each point of $\N$.   The real symmetric part of 
$\omega$, viewed as a symmetric $2$-tensor, describes the infinitesimal change
in the metric induced by the infinitesimal deformation corresponding to $\omega$. 

One can always choose a representative for a cohomology class where 
$\eta$ is {\em symmetric}, when viewed as a section of the bundle ${\rm Hom}\,(T\N,T\N)$.
To do this, one notes that, since $\omega$ is $d_E$-closed, it is the image of a 
{\em locally defined} section.  Then $\eta$ is symmetric if the local section has the property 
that its imaginary part equals $ \frac 1 2$ of the curl of the real part, where both the 
real and imaginary parts are viewed as locally defined vector fields.  It is shown in
\cite{HK1} that such a choice of local section is always possible; in that paper, such
a local section was called a {\em canonical lift} of the real part. (Note, however, that
the definition of curl in that paper differs from the standard one, which is the one used
here, by a sign and a factor of $2$).

If $\omega$ is harmonic then $\eta$ satisfies the equation
$$(D^*D +  D D^*) \eta ~+~ \eta ~=~  0$$
where $D$ denotes the exterior covariant derivative on vector valued $1$-forms and 
$D^*$ is its adjoint.  If $\eta$ is also traceless then it satisfies $D^* \eta = 0$; hence, 
it satisfies the simpler equation
\begin{eqnarray}D^*D \eta ~+~ \eta ~=~ 0.\label{etaeqn} 
\end{eqnarray}

In this case we have that $\tilde \eta = \hat * D\eta$, where $\hat *$ is the Hodge star 
operator on forms in $\N$ and takes the vector valued $2$-form $D\eta$ to a vector 
valued $1$-form.  It is also true in this case that $\hat * D\eta$ is symmetric and traceless.  
Thus, we can write
\begin{eqnarray} \omega ~=~ \eta ~+~ \tilde \eta ~=~ \eta ~+~ i \,\hat * D\eta,\label{omegadecomp}
\end{eqnarray}
where both $\eta$ and $\hat * D\eta$ are symmetric and traceless.

Let $\N$ be a hyperbolic 3-manifold with tubular boundary, whose boundary components 
are tori of tube radii $R_1, \ldots , R_k$. The boundary components will
always be oriented by the {\em inward} normal for $\N$.
For any $T\N$-valued $1$-forms $\alpha, \beta$ we define
\begin{eqnarray}\b_R(\alpha,\beta) = \int_{\bd M} \hat * D\alpha \wedge \beta,
\label{b_defn}\end{eqnarray}
where $R$ denotes the vector $(R_1, \ldots, R_k)$.

In this integral, $\hat * D\alpha \wedge \beta$ denotes
the real valued 2-form obtained
using the wedge product of the form parts, and the geometrically
defined inner product on the vector-valued parts of the
$TM$-valued $1$-forms $\hat *\!D\alpha$ and $\beta$.

Returning to equation (\ref{etaeqn}), we take the $L^2$ inner product on $\N$ of
this equation with $\eta$ and integrate by parts.  We then obtain the following
Weitzenb\"ock formula with boundary for any harmonic infinitesimal deformation
$\omega$ of the form (\ref{omegadecomp}):
\begin{eqnarray}||D\eta||^2 ~+~ ||\eta||^2 ~=~ \b_R(\eta,\eta).\label{bdryeq}
\end{eqnarray}

In particular, for a non-trivial infinitesimal deformation, the boundary term $\b_R(\eta,\eta)$ 
must be {\em positive}.  The proof in \cite{HK1} that there are no infinitesimal deformations
of hyperbolic cone manifolds (with cone angles at most $2 \pi$) fixing the cone angles
amounts to showing that, for a deformation fixing the cone angles, a harmonic
representative $\omega$ can be found so that this boundary term goes to $0$
as the tube radius goes to zero.  The results in the current paper depend on showing
that when the Dehn surgery coefficients are all preserved (and the tube
radii are all bigger than a universal constant),  a harmonic representative can be 
found so that this boundary term is {\em non-positive}.

Recall that the harmonic form $\omega$ will be found by starting with a
representative $\hat \omega$ which equals a standard form $\omega_0$ in a neighborhood
of the boundary, hence is harmonic in that neighborhood, but not globally harmonic.
Solving equation (\ref{correction}) provides a correction term $d_E s$ that is added to
$\hat \omega$ to make it globally harmonic.
We denote the correction term in a neighborhood of the boundary by
$\omega_c$ and decompose the harmonic
representative $\omega$ as $\omega = \omega_0+ \omega_c$ in that neighborhood.
Note that, since $ \omega_0$ is harmonic, the correction term $\omega_c$ will 
also be harmonic in that neighborhood.

The standard forms are all of the form (\ref{omegadecomp}) and their real parts
satisfy equation (\ref{etaeqn}).  Thus, being able to write $\omega$ in
this form is equivalent to being able to solve (\ref{correction}) in such a way 
that $\omega_c$ can be written in this form.
Then, in a neighborhood of the boundary, 
we write the real part of
$\omega$ as the sum of the real part of $\omega_0$ and that of $\omega_c$, 
$\eta = \eta_0 + \eta_c$.   Both $\eta_0$ and $\eta_c$
satisfy equation (\ref{etaeqn}) in that neighborhood.

Using this decomposition of $\eta$ on the boundary, we can try to compute
the boundary term $b_R(\eta,\eta)$.
In \cite{HK2} we saw that 
the cross-terms vanish so that
the boundary term is simply the sum of two boundary terms:
\begin{eqnarray} \label{bdrydecomp} \b_R(\eta,\eta) = \b_R(\eta_0,\eta_0)
+ \b_R(\eta_c,\eta_c).
\end{eqnarray}

Since the standard forms are quite explicit, it is fairly easy to find conditions
under which the term $\b_R(\eta_0,\eta_0)$ is non-positive and to estimate
its value in general.  Thus, we finally come to the boundary value problem
we wish to solve:

\hfill
\eject

\noindent {\bf Boundary Value Problem:}
{\em Find boundary conditions on the real-valued $1$-form $\tau$ so that
there is always a unique solution to equation (\ref{taueqn}) when $\zeta$ is
smooth on all of $\N$, including the boundary.  Furthermore,
these boundary conditions
must ensure that $\omega_c$ satisfies (\ref{omegadecomp}), hence that
$\eta_c$ satisfies equation (\ref{etaeqn}).  Finally, the boundary term, 
$\b_{R}(\eta_c,\eta_c)$,
in the Weitzenb\"ock formula (\ref{bdrydecomp}) must always be non-positive}.

\section{The Weitzenb\"ock correction term}\label{weitz}

In this section we derive a general formula for a boundary integral
(see (\ref{bdef}) below)
which we refer to as the {\em \W correction term}, that arises in
the \W formula (\ref{bdryeq}) for harmonic infinitesimal deformations of 
a compact hyperbolic $3$-manifold with boundary.  We then specialize
to the special case of interest in this paper, when the boundary is tubular.
We further compute the boundary term in this case under the hypothesis of 
specific boundary conditions.  In the next section we show that such
boundary conditions can always be realized.

Let $M$ be an oriented compact hyperbolic 3-manifold with boundary $\bd M$
and let $E$ be the bundle of (germs of) infinitesimal isometries on $M$.
We denote by $d_E$ the coboundary operator on smooth $E$-valued $i$-forms on $M$;
the latter are denoted by $\Omega^i(M;E)$.
This operator satisfies the equation $d_E^2 = 0$ and $H^1(M;E)$, the first cohomology
of $M$ with coefficients in $E$, is defined to be the $d_E$-closed $E$-valued
$1$-forms modulo those of the form $d_E s$ where $s$ is an $E$-valued $0$-form;
i.e., a global section of $E$.  This cohomology group represents the
(scheme of) infinitesimal hyperbolic deformations of $M$.

As discussed in the previous section, a boundary integral, $b_R(\eta,\eta)$, occurs in
the \W formula that holds for a class of harmonic ($d_E$-closed and co-closed) $E$-valued
$1$-forms.  In the case when $M$ has tubular boundary, these harmonic forms 
are constructed by adding a coboundary of the form $d_E s$ to a representative in
$H^1(M;E)$ which is in a standard form near the boundary.  In particular, 
we are interested in the contribution to the boundary integral coming from
this coboundary.  Because of the decomposition (\ref{bdrydecomp}) 
of the boundary integral, this contribution can be computed as a boundary
integral involving only the $E$-valued $1$-form $d_E s$. 
In this section we compute this boundary integral, on a general compact 
hyperbolic $3$-manifold with boundary, for any 
$E$-valued $1$-form that is of the form $d_E s$.

In the previous section we observed that the bundle $E$ can be identified
with the complexified tangent bundle $TM \otimes \C$. Then, a global
section $s$ of $E$ can be written as  $s = u +  \tilde u \, i$ where
$u,\tilde u$ are global vector fields on $M$.  Similarly we can decompose 
$d_E s \in \Omega^1(M;E)$ into its real and imaginary parts 
$d_E s = \eta_c + \tilde \eta_c i$. Both $\eta_c$ and  $\tilde \eta_c$ 
are vector field valued $1$-forms; i.e., elements of $\Omega^1(M;TM)$. 
They can equivalently be viewed as elements of ${\rm Hom} \,(TM,TM)$.
In \cite{HK1} we computed that
\begin{eqnarray}d_E s = \eta_c + \tilde \eta_c i = 
(Du -\rot_{\tilde u}) + ( D\tilde u + \rot_u)i
\label{d_Es}
\end{eqnarray}
where $Du \in \Omega^1(M;TM)$ is the covariant derivative of
the vector field $u$
and $\rot_u \in {\rm Hom} \,(TM,TM)$ at any point $p \in M$ 
is the infinitesimal rotation determined by the tangent vector
$u(p) \in T_p M$. Thus for each tangent vector $X \in TM$,
\begin{eqnarray}Du(X) = \nabla_X u \text{~~~and~~~} \rot_u(X)= u \times X,
\label{DuX}\end{eqnarray}
where $\times$ denotes the cross product defined by the 
Riemannian metric and orientation on $M$.

We now define the boundary integral of interest to us.  For any element
$d_E s = \eta_c + \tilde \eta_c i\in \Omega^1(M;E)$ we define the 
{\it \W correction term} by   
\begin{equation} \label{bdef} b = \int_{\bd M} \eta_c \wedge \tilde\eta_c,
\end{equation}
where the boundary is oriented with respect to the inward normal.
Of course, this integral can be defined for any element of $\Omega^1(M;E)$,
decomposed into its real and imaginary part.  However, we will only be interested 
in this section in computing it for those elements which are coboundaries;
hence, the name ``correction term".  Much of our computation is valid for
any such element, but we will then specialize to the case where both
$\eta_c$ and $\tilde \eta_c$ are symmetric which is the case that will
arise during the process of finding a harmonic representative discussed
in the previous section.  In that situation we will also have the relation
$\tilde \eta_c = \hat* D\eta_c$ from which it follows immediately that
the \W correction term equals $- ~ b_R(\eta_c, \eta_c)$, where $b_R$
is defined by (\ref{b_defn}).  As discussed in the previous section,
because of the \W formula (\ref{bdryeq}) and its decomposition (\ref{bdrydecomp})
we will be interested in finding boundary conditions on $s$ that will
guarantee that $ b_R(\eta_c, \eta_c)$ is non-positive.
Thus, we will be interested in conditions that will imply that 
the \W correction term is \emph{non-negative}.

In order to compute this boundary integral, it is useful to decompose
sections and forms into their tangential and normal parts near the boundary.
Specifically, the surfaces equidistant from $\bd M$ give a foliation in a 
neighborhood of $\bd M$ in $M$, and there is a unit vector field $\n$ consisting
of normal vectors to these equidistant surfaces pointing
inwards from $\bd M$.
The vector field $u$ on $M$ can be decomposed
near $\bd M$ as
$$u = v + h \n$$
where $v$ is the component tangent to the equidistant surfaces
and $h = u \cdot n$ is the component in direction of the normal $\n$.
(We use $\cdot$ to denote the Riemannian inner product on $M$.)
Similarly, we write
$$\tilde u = \tilde v + \tilde h \n.$$

As in (\ref{d_Es}) above we write $d_E s = \eta_c + \tilde \eta_c i$ where
$s = u + \tilde u i$ is a global section, decomposing both $s$ and $d_E s$ into their real and imaginary parts.  
We are only interested in the values of 
$\eta_c$ and $\tilde \eta_c$ restricted to $\bd M$.
Viewed as a $TM$-valued $1$-form there, for each $X \in T_p(\bd M)$ we decompose $\eta_c(X)$ 
into a tangential part $G(X)$ and normal part $F(X)$.  We can then write
$$\eta_c(X) = G (X) + F(X) \n, \text{ where } G(X) \in T_p(\bd M),~~ F(X) \in
\R.$$
Similarly we decompose $\tilde \eta_c(X)$ into a tangential part $\tilde G(X)$
and a normal part $\tilde F(X)$.

Finally, let $S : T(\bd M) \to T(\bd M)$
denote the {\em shape operator} 
defined by $S(X) = \nabla_X n$
where $n$ is the {\em inward} unit normal to $\bd M$.
Then $S$ is a self-adjoint operator whose 
eigenvalues are the principal
curvatures of $\bd M$, $\tr S = H$ is the mean curvature of $\bd M$
and $\det(S) = K_{ext}$ is the extrinsic curvature of $\bd M$.
(Note that with our sign convention for $S$, the principal curvatures are positive when $\bd M$ is concave.)

With this notation established, we make the following computation:

\begin{lemma} \label{FG}
Let $\sigma, \tilde \sigma$ be the $1$-forms on $\bd M$ dual
to the vector fields $v, \tilde v$ on $\bd M$. Then, using the
notation defined above,
$G,\tilde G : T_p (\bd M) \to T_p (\bd M)$
and $F, \tilde F : T_p (\bd M) \to \R$ are given by
\begin{equation}
G = \bar D v + h S - \tilde h J, \qquad
\tilde G = \bar D \tilde v + \tilde h S + h J
\end{equation}
and
\begin{equation}
F = dh - S \sigma - * \tilde \sigma , \qquad
\tilde F = d \tilde h -S \tilde \sigma + * \sigma
\end{equation}
where $\bar D$ is the exterior covariant derivative on $\bd M$,
$S$ is defined on $1$-forms $\sigma$ by
$S\sigma(X)=Sv \cdot X =\sigma(SX)$, 
and $J: T(\bd M) \to T(\bd M)$ is the rotation by $\pi/2$ given by
$JX = \n \times X$ (so $J^2 = - \text{identity}$).
\end{lemma}

\begin{proof}
For $X \in T(\bd M)$ we have, using equations (\ref{d_Es}) and (\ref{DuX}), 
\begin{eqnarray*}
\eta_c(X) &=& \nabla_X u  - \tilde u \times X \\
&=& \nabla_X( v + h \n) - ( \tilde v + \tilde h \n) \times X \\
&=& \nabla_X v + h \nabla_X \n +X(h) \n - \tilde v \times X- \tilde h \n
\times X .
\end{eqnarray*}
Thus the tangential part $G : T_p (\bd M) \to  T_p (\bd M)$  is given by
$$G(X) = \bar\nabla_X v + h S(X) - \tilde h \n \times X
= (\bar D v + h S - \tilde h J)(X)$$
where $\bar \nabla, \bar D$ 
denote the
Riemannian connection  and exterior covariant derivative on $\bd M$.
Further
the normal component $F : T(\bd M) \to \R$ is given by
\begin{eqnarray*}
F(X) = \eta_c(X) \cdot \n &=&  X(h) - \tilde v \times X \cdot \n + (\nabla_X
v) \cdot n \\
&=& \grad h \cdot X - \n \times \tilde v \cdot X - v \cdot \nabla_X n\\
&=&( \grad h - J\tilde v - S(v)) \cdot X \\
&=& (dh - * \tilde\sigma - S\sigma)(X).
 \end{eqnarray*}

Similarly, we find $\tilde\eta_c = D \tilde u + \rot_u$ 
has tangential part $\tilde G : T_p (\bd M) \to  T_p (\bd M)$
given by
$$\tilde G(X) = \bar\nabla_X \tilde v + \tilde h S(X) +  h \n \times X
= (\bar D \tilde v + \tilde h S + h J)(X) $$
and normal part $\tilde F : T(\bd M) \to \R$ given by
$$\tilde F(X) = X(\tilde h) + v \times X \cdot \n + (\nabla_X \tilde v) \cdot n
= (d \tilde h + * \sigma - S \tilde \sigma)(X).$$
\end{proof}

We can view $F, \tilde F$ and $G, \tilde G$ as real-valued and
vector-valued $1$-forms on $\bd M$.  It is then possible to 
define the wedge products
$F\wedge \tilde F$ and $G \wedge \tilde G$, where in the latter case we also 
use the dot product on $TM$ from the hyperbolic metric on $M$ to obtain
a real-valued $2$-form.  Then our boundary term can be expressed in
terms of these wedge products as
$$b = \int_{\bd M} \eta_c \wedge \tilde\eta_c =
\int_{\bd M} F\wedge\tilde F +G\wedge\tilde G.$$

We will now compute the two summands in this expression separately.  
First recall that if $\omega$ is a 1-form on a Riemannian 3-manifold $M$,
then its exterior derivative $d\omega$ satisfies 
\begin{equation} \label{extd1}
d\omega(X,Y) = X \omega (Y) - Y \omega(X) - \omega([X,Y])
\end{equation}
for all vector fields $X,Y$ on $M$. We can also rewrite this using 
covariant derivatives:
\begin{equation} \label{extd2}
d\omega(X,Y) = \nabla_X \omega \,(Y) - \nabla_Y \omega \,(X)
\end{equation}
since $\nabla_X Y - \nabla_Y X = [X,Y]$.

To analyze the boundary term $\int_{\bd M} G \wedge \tilde G$
we will use the following.

\begin{lemma} Let $v,\tilde v$ be vector fields on $\bd M$ with dual 1-forms $\sigma,\tilde\sigma$, and let
$dA$ denote the area $2$-form on $\bd M$. Denote by $\delta$ the adjoint of 
the exterior derivative on $\bd M$. Then
\begin{enumerate}
\item $J \wedge S = \tr S \, dA = H \, dA$ where $H$ is the mean curvature
of $\bd M$,  
\item $J \wedge \bar Dv = -*\delta \sigma$,    
\item $\bar Dv \wedge S  =d(S\sigma)$,     
\item $S\wedge S=0$,
\item $J\wedge J = 0$,
\item $\int_{\bd M} \bar Dv \wedge \bar D\tilde v 
= \la K \sigma, *\tilde \sigma \ra$,
where $K$ is the Gaussian curvature of $\bd M$
and  $\la \cdot, \cdot \ra$ is the $L^2$ inner product on $\bd M$.
\end{enumerate}
\end{lemma}

\begin{proof} Let $e_1, e_2$ be an oriented orthonormal basis for $T_p (\bd M)$.
Then for any linear operator $L : T_p (\bd M)\to T_p (\bd M)$ we have
\begin{eqnarray*}
(J\wedge L) (e_1, e_2) &=& J(e_1) \cdot L(e_2) - J(e_2) \cdot L(e_1) \\
&=& e_2 \cdot L(e_2) + e_1 \cdot L(e_1) = \tr L.
\end{eqnarray*}
Hence $J \wedge S = \tr S \, dA = H \, dA$ and
$J \wedge \bar Dv = \tr \bar Dv \, dA = \div v \, dA = -*\delta \sigma$.

Next we prove part (3). For $X,Y \in T_p(\bd M)$ we have
\begin{eqnarray*}
(\bar Dv \wedge X) (X,Y) &=& \nabla_X v \cdot S(Y) - \nabla_Y v \cdot S(X) \\
&=&  II(\nabla_X v ,Y) - II(\nabla_Y v,X) \\
&=& X II(v ,Y) - II(v ,\nabla_X Y) - (\nabla_X II)(v ,Y) \\
&& \quad - \left( Y II(v ,X) - II(v ,\nabla_Y X) - (\nabla_Y II)(v ,X) \right)
\end{eqnarray*}
where $II(X,Y)=S(X) \cdot Y = X \cdot S(Y)$ is the second fundamental form.
But $(\nabla_X II)(v ,Y) = (\nabla_Y II)(v,X) $ by the Codazzi-Mainardi
equations for a space of constant curvature (see, for example,
\cite[chap.1, Thm 11 and Cor 12]{spivak}). Hence
\begin{eqnarray*}
(\bar Dv\wedge S) (X,Y) &=&  X II(v ,Y) - Y II(v ,X) - II( v, [X,Y]) \\
&=& X (Sv \cdot Y) - Y (Sv \cdot X) - Sv \cdot[X,Y] \\
&=& X(S\sigma(Y))-Y(S\sigma(X))-S\sigma([X,Y]) \\
&=& d(S\sigma)(X,Y).
\end{eqnarray*}

Finally, we have $S\wedge S=0$ and $J\wedge J = 0$ by the skew-symmetry
of the wedge product, and for vector fields $X,Y$ on $\bd M$ we have
$$D^2 v(X,Y) =\bar{\nabla}_X \bar{\nabla}_Y v -  \bar{\nabla}_Y \bar{\nabla}_X v
-\bar{\nabla}_{[X,Y]} v =\bar R(X,Y) v,$$
where $\bar R$ is the Riemann curvature tensor on $\bd M$.
But $\bar R(e_1,e_2)$ is infinitesimal rotation by $-K$, 
hence
$D^2 v = - K J v \, dA$ where $dA$ is the area 2-form on $\bd M$ and
$$\int_{\bd M} \bar Dv \wedge \bar D\tilde v =
\int_{\bd M}d(v \wedge \bar D\tilde v) -v \wedge D^2 \tilde v= 
\int_{\bd M}  K v \cdot J \tilde{v} \, dA 
= \la K \sigma, * \tilde \sigma \ra$$
by Stokes' theorem. This completes the proof of the Lemma.
\end{proof}

Using this result and Lemma \ref{FG} we obtain
\begin{lemma} \label{GwedgeG}
$$\int_{\bd M} G \wedge \tilde G =
\la \tilde h , *d{S\sigma} \ra
 - \la h , * d {S\tilde \sigma} \ra
+ \la h , \delta \sigma \ra + \la \tilde h , \delta \tilde \sigma \ra
-( \la H h, h \ra + \la H \tilde h , \tilde h \ra) +  \la K \sigma, * \tilde{\sigma} \ra ,
$$
where $\la \alpha , \beta \ra = \int_{\bd M} \alpha \wedge *\beta$ 
denotes the $L^2$ inner product on $\bd M$.
\end{lemma}

Next we study the boundary term $\int_{\bd M} F \wedge \tilde F = - \la F, *\tilde F \ra$. 
From Lemma \ref{FG} we have
$$F = dh - S \sigma - * \tilde \sigma \qquad{\rm and }\qquad
\tilde F = d \tilde h -S \tilde \sigma + * \sigma.$$
Using this we obtain
\begin{eqnarray*}
\int_{\bd M} F \wedge \tilde F 
&=& - \la dh - S \sigma - * \tilde \sigma, *d \tilde h -*S \tilde \sigma - \sigma \ra \\
&=&  - \la dh,*d \tilde h \ra + \la dh, *S \tilde \sigma + \sigma \ra
+\la S \sigma + * \tilde \sigma, *d\tilde h \ra
-\la S \sigma + * \tilde \sigma, *S \tilde \sigma + \sigma \ra.
\end{eqnarray*}

But the first term vanishes since
$$-\la dh, *d \tilde h \ra = \int_{\bd M} dh \wedge d\tilde h  = 
\int_{\bd M} d(h \wedge d\tilde h) =0,$$
so we obtain

\begin{lemma}
$$
\int_{\bd M} F \wedge \tilde F =
\la dh, *S \tilde \sigma + \sigma \ra
+\la S \sigma + * \tilde \sigma, *d\tilde h \ra
-\la S \sigma + * \tilde \sigma, *S \tilde \sigma + \sigma \ra.
$$
\end{lemma}

Combining the previous results gives the following
\begin{theorem} \label{generalweitzbdry} Let $s = u + i \, \tilde u$ be a section of the bundle $E$ and 
let $d_E s = \eta_c + i \tilde \eta_c$ denote its image under the coboundary operator 
$d_E s$.  Suppose that $u = v + h n$, $\tilde u = \tilde v + \tilde h n$ are
the decompositions into tangential and normal parts of the vector fields $u, \tilde u$.  
Denote by $\sigma, \tilde \sigma$ the $1$-forms on that are dual on $\bd M$ 
to the vector fields $v, \tilde v$, respectively.  Then the Weitzenb\"ock boundary term 
equals
\begin{eqnarray*}
b=\int_{\bd M} \eta_c \wedge \tilde \eta_c
&=& 2( \la h, \delta \sigma \ra + \la \tilde h, \delta \tilde \sigma \ra)
-\la S \sigma + * \tilde \sigma, *S \tilde \sigma + \sigma \ra \\
&& -( \la H h, h \ra + \la H \tilde h , \tilde h \ra) +  \la K \sigma, * \tilde{\sigma} \ra .
\end{eqnarray*}
\end{theorem}

\begin{proof} From the previous lemmas we have
\begin{eqnarray*}
\int_{\bd M} \eta_c \wedge \tilde \eta_c &=& \int_{\bd M} F\wedge\tilde F +G\wedge\tilde G \\ 
&=&  \la dh, *S \tilde \sigma \ra + \la dh, \sigma \ra
+\la S \sigma, *d\tilde h \ra  +  \la* \tilde \sigma, *d\tilde h \ra \\
&& -\la S \sigma + * \tilde \sigma, *S \tilde \sigma + \sigma \ra \\
&& 
 - \la h , * d {S\tilde \sigma} \ra
+ \la h , \delta \sigma \ra + \la \tilde h , *d{S\sigma} \ra +  \la \tilde h , \delta \tilde \sigma \ra \\
&& -( \la H h, h \ra + \la H \tilde h , \tilde h \ra) +  \la K \sigma, * \tilde{\sigma} \ra .
\end{eqnarray*}
We can simplify this sum by noting that
$\la dh, \sigma \ra = \la h, \delta \sigma \ra$, 
$\la *\tilde\sigma, *d\tilde h \ra = \la \tilde h, \delta \tilde\sigma \ra$,
 $\la dh, *S\tilde \sigma \ra -\la h, *dS\tilde \sigma \ra   = 0$, and
$ \la S \sigma, *d \tilde h\ra + \la \tilde h, *dS\sigma\ra = 0$.
This gives the result.
\end{proof}

The computation of the boundary term in Theorem \ref{generalweitzbdry} is
valid for general elements $\eta_c, \tilde \eta_c \in \Omega^1(M;TM)$ that
are of the form $d_E s = \eta_c + \tilde \eta_c i$.  However, this boundary term
is primarily of interest when it comes from the \W formula (\ref{bdryeq}), 
as discussed in the previous 
section.  Specifically, we are interested in the case when, in a neighborhood 
of the boundary, both $\eta_c, \tilde \eta_c$
are symmetric and traceless when viewed as elements of ${\rm Hom} \,(TM,TM)$.  
Writing $s = u + i \tilde u$ as before, the condition that $\eta_c$ and $\tilde \eta_c$
are symmetric is equivalent to the equations (derived in \cite[Section 2]{HK1}):
\begin{equation}
2 \tilde u = \curl u, ~ 2 u = - \curl \tilde u \label{3dcurleqn}
\end{equation}

Here we view the curl of a vector field
in $3$-dimensions as itself being a vector field.  Then, on the boundary of $M$, 
the normal component of the $3$-dimensional curl of $u$ is just the (scalar) $2$-dimensional curl of $v$, the tangential part of $u$.  A similar statement holds
for the normal component of $\curl \tilde u$.  Since the normal components of
$u, \tilde u$ equal $h, \tilde h$, respectively, we obtain:
\begin{equation}
2 \tilde h = \curl v = *d \sigma \qquad{\rm ~and~} \qquad 2h = - \curl \tilde v = - * d \tilde \sigma,
\label{2dcurleqn}
\end{equation}
where $\sigma, \tilde \sigma$ are the $1$-forms on $\bd M$ dual to $v, \tilde v$, 
respectively, and $d$ denotes the exterior derivative operating on forms on $\bd M$.

Note that the equations (\ref{3dcurleqn}) only hold in a neighborhood of the boundary.
However, since all our computations are local to the boundary, this will suffice.
It turns out that (\ref{3dcurleqn}) also implies that $\eta_c$ and $\tilde \eta_c$ 
are traceless (using $\div \curl = 0$), but we will not use this in our computations.

Using (\ref{2dcurleqn}) we can rewrite $b$ in terms of $\sigma$ and $\tilde \sigma$.
We compute
$$2\la h, \delta \sigma \ra = -\la *d\tilde \sigma, \delta \sigma \ra = 
- \la \delta(*\tilde \sigma), \delta \sigma \ra= - \la *\tilde \sigma, d\delta \sigma \ra,$$
$$2\la \tilde h, \delta \tilde \sigma \ra = \la *d \sigma, \delta \tilde \sigma \ra = 
- \la d \sigma, d(*\tilde \sigma) \ra = - \la \delta d \sigma, * \tilde \sigma \ra,$$
and note that
$$\la *d\tilde\sigma,*d\tilde\sigma \ra=
\la \delta(*\tilde\sigma),\delta(*\tilde\sigma) \ra, ~~~
\la *d\sigma,*d\sigma \ra= \la d\sigma,d\sigma \ra.$$

Then we obtain
\begin{eqnarray} \label{beqn}
b &=& -\la *\tilde \sigma, d \delta \sigma  + \delta d \sigma \ra
-\la S \sigma + * \tilde \sigma, * S \tilde \sigma  + \sigma \ra \\
\nonumber && -\frac{1}{4}(\la H \delta(*\tilde\sigma),\delta(*\tilde\sigma) \ra
 + \la H d\sigma,d\sigma \ra ) +  \la K \sigma, * \tilde{\sigma} \ra.
 \end{eqnarray}

Equations (\ref{3dcurleqn}) provide relations between $u$ and $\tilde u$,
equations (\ref{2dcurleqn}) coming from the normal component of those relations.
Similarly, the tangential component of (\ref{3dcurleqn}) implies that $\tilde v$, 
hence $\tilde \sigma$, can be expressed in terms of $u$ and its derivatives in
a neighborhood of the boundary.  In the next section we will define boundary
conditions that will allow us to express $\tilde \sigma$ {\emph on $\bd M$} 
purely in terms of $\sigma$ and its tangential derivatives.  In particular,
the expression will not involve the normal component $h$ or normal 
derivatives of $\sigma$.

There are many possible boundary conditions of this sort and we denote
by $A$ a general linear differential operator on $1$-forms on $\bd M$.
Then, as a simplifying notational device, we can express the relation between 
$\tilde \sigma$ and $\sigma$ on $\bd M$ as:
\begin{equation}
\tilde \sigma = *A\sigma \label{Adef}
\end{equation}

Finally, it is useful to define a linear operator $\hat S$ on 1-forms by $\hat S = - * S *$, 
so that $*S\tilde \sigma = \hat S * \tilde \sigma$.  It is easy to check that this 
operator satisfies
$$\hat S S = \det(S) I = K_{ext} I,$$ 
where $K_{ext}$ is the extrinsic curvature of $\bd M$.  It follows, since
the curvature of the ambient hyperbolic manifold equals $-1$, that
$$\hat S S - I = (K_{ext}-1) I = K_{int} I = K I.$$
As before $K$ denotes the intrinsic Gaussian curvature of $\bd M$.

Rewriting equation (\ref{beqn}) using the operators $\hat S$ and $A$ we
find
 \begin{eqnarray*}
b &=& \la A \sigma, \Delta \sigma \ra +
\la (A - S) \sigma, (I - \hat S A) \sigma \ra\\
&&-\frac{1}{4}(\la H \delta A\sigma,\delta A\sigma \ra
 + \la H d\sigma,d\sigma \ra)
 - \la K \sigma, A \sigma \ra.
 \end{eqnarray*}

Hence we obtain
\begin{theorem} \label {harmonicweitzbdry}
If  $2 \tilde h = \curl v = *d \sigma$,  $2h = - \curl \tilde v = - * d \tilde \sigma$, 
$\hat S = - *S*$
and $\tilde \sigma = *A\sigma$,
then the Weitzenb\"ock boundary term is given by
  \begin{equation}
b ~=~ \la (A - S) \sigma, (I - \hat S A) \sigma \ra
+ \la A \sigma, (\Delta-K) \sigma \ra
-\frac{1}{4}(\la H \delta A\sigma,\delta A\sigma \ra
 + \la H d\sigma,d\sigma \ra).\label{bwithA}
 \end{equation}
\end{theorem}

We observe that the expression for the boundary term in Theorem \ref{harmonicweitzbdry}
can be viewed as a quadratic form on $1$-forms $\sigma$.  Except for the operator $A$,
the basic terms in this quadratic form come from the geometry of the boundary of $M$.
In particular, $S,K,H$ are the shape operator, Gaussian curvature and mean curvature
of $\bd M$.  So, given the manifold $M$ with its boundary $\bd M$, the only
flexibility we have on this quadratic form is the tangential operator $A$.  We can
attempt to control this operator by our choice of boundary conditions when solving
equation (\ref{taueqn}).  Recall that the $1$-form $\tau$ in (\ref{taueqn}) 
is dual to the vector field $u$ which in turn determines $\tilde u$ by (\ref{3dcurleqn}).

Our goal is to find boundary conditions which determine an operator $A$ with the
property that this quadratic form is positive semi-definite; i.e., so that
the boundary term (\ref{bwithA}) is non-negative for all $\sigma$.

We now specialize to the case when $M$ has {\em tubular boundary}.  Then 
each boundary component is topologically a torus,
and the principal curvatures $k_1, k_2$ are constant so that 
the mean curvature $H = k_1 + k_2$ is constant. The extrinsic curvature is $K_{ext} = k_1 k_2 = 1$ 
and the intrinsic curvature is $K=0$; i.e., the torus is flat.  The operator $\hat S$ 
equals $S^{-1}$ in this case and both $S$ and $\hat S$ are parallel. 
Then the boundary term simplifies to
  \begin{eqnarray} \label{flatbdryeqn}
b &=&  \la (A - S) \sigma, (I - \hat S A) \sigma \ra
+ \la A \sigma, \Delta \sigma \ra \\
\nonumber && -\frac{H}{4}(\la \delta A\sigma,\delta A\sigma \ra
 + \la d\sigma,d\sigma \ra).
 \end{eqnarray}
 
 If we denote by $A_0$ the $0$-th order part of the operator $A$ (i.e., the
 part that involves taking no derivatives), then the $0$-th order part of
 this quadratic form is simply
   \begin{eqnarray*} 
   \la (A_0 - S) \sigma, (I - \hat S A_0) \sigma \ra = 
   - \la (A_0 - S) \sigma, S^{-1} (A_0 - S) \sigma \ra,
   \end{eqnarray*}
 where we have used that fact that $\hat S = S^{-1}$ to obtain the second expression.  
 Since $S^{-1}$ is a positive operator this quantity is non-positive.  Our
 only hope of having a non-negative quadratic form is to choose $A$ so that
 $A_0 = S$.
 
 The computations below show that the choice of $A = S$ in fact does lead
 to a non-negative quadratic form.  However, as will be discussed in the next section, 
 we have been unable to find an elliptic boundary value problem that leads to
 this value of $A$.  Nevertheless, we are able to find such a boundary value
 problem that leads to a slightly perturbed value of $A$ that still defines a
 non-negative quadratic form.  
 
 Suppose the tangential operator $A$  equals
 \begin{equation} A=S+\frac{\varepsilon}{2}\delta d,\label{Avalue} \end{equation}
 where $\varepsilon>0$ is a constant. 
 Thus, in the above Weitzenb\"ock boundary term we have
$A-S = \frac{\varepsilon}{2}\delta d$ and 
$I-\hat S A =I-\hat S S - \frac{\varepsilon}{2}\hat S\delta d= - \frac{\varepsilon}{2}\hat S\delta d.$
 
 In the next section we will show that it is always possible to solve
 equation (\ref{taueqn}) in such a way that the $1$-forms $\sigma, \tilde \sigma$
 arising from the solution satisfy the relation (\ref{Adef}) with this
 value of $A$. 
 For now, we will assume that this can be done and complete the computation
 of the boundary term with this value of $A$.

Now $\delta A = \delta S+\frac{\varepsilon}{2} \delta \delta d = \delta S$
and $\la \delta d \sigma, d \delta \sigma \ra = \la d \sigma, d d \delta \sigma \ra =0$,
so equation (\ref{flatbdryeqn}) becomes
 \begin{eqnarray*}
b = \int_{\bd M} \eta_c \wedge \tilde \eta_c
&=& \la S \sigma, \Delta \sigma \ra
 - \frac{H}{4}\left( \la \delta S \sigma, \delta S \sigma \ra + \la d\sigma, d \sigma \ra \right) \\
 && \quad -\left(\frac{\varepsilon}{2}\right)^2 \la \delta d \sigma, \hat S \delta d \sigma \ra
 + \frac{\varepsilon}{2} \la \delta d \sigma, \delta d  \sigma \ra.
\end{eqnarray*}
and we want to find geometric conditions on $\bd M$ guaranteeing that this
boundary term is non-negative. 

 First we consider $\la \delta S \sigma, \delta S \sigma \ra + \la d\sigma, d \sigma \ra$
 and  $\la S \sigma, \Delta \sigma \ra$.  
Since the metric on $\bd M$ is Euclidean, we can choose a parallel orthonormal
 frame field $e_1, e_2$ on $\bd M$ consisting of eigenvectors of $S$
 with eigenvalues $k_1, k_2$  at every point. 
 Let $\theta_1, \theta_2$ be the dual 1-forms on $\bd M$ and
 write $\sigma = \s_1 \theta_1 + \s_2 \theta_2$. 
 Now, using equation (\ref{extd2}), 
 $$d \sigma = (\nabla_1 \s_2 - \nabla_2 \s_1)\,dA \qquad{\rm and }
 \qquad -\delta(S\sigma) = (k_1 \nabla_1 \s_1 + k_2 \nabla_2 \s_2) \, dA$$
 where  $\nabla_i = \bar \nabla_{e_i}$ for $i = 1,2$.  Hence
\begin{eqnarray*}
(*\delta S \sigma)^2 +(*d\sigma)^2
&=& k_1^2 (\nabla_1 \s_1)^2 + k_2^2 (\nabla_2 \s_2 )^2  
 +  ( \nabla_1 \s_2 )^2 +  ( \nabla_2 \s_1 )^2 \\
 && \quad  
 + 2\left(k_1 k_2 (\nabla_1 \s_1) ( \nabla_2 \s_2) - (\nabla_1 \s_2) (\nabla_2 \s_1)\right) .
 \end{eqnarray*}
 Since $k_1 k_2 =1$ the bracketed terms become
 $$\nabla_1 \s_1 \nabla_2 \s_2 - \nabla_1 \s_2 \nabla_2 \s_1
 = d\s_1 \wedge d\s_2(e_1,e_2),$$
 and their integral over $\bd M$ is
 $$\int_{\bd M} d\s_1 \wedge d\s_2 = \int_{\bd M} d(\s_1 \wedge d\s_2) =0.$$
 Hence
 $$\| \delta S \sigma \|^2 + \| d\sigma \|^2 = 
  k_1^2 \|\nabla_1 \s_1\|^2 + k_2^2 \| \nabla_2 \s_2 \|^2 
 +  \| \nabla_1 \s_2\|^2 + \| \nabla_2 \s_1\|^2 $$
 where $\| \cdot \|$ is the $L^2$-norm on $\bd M$.
Using integration by parts,
 $$\la S \sigma, \Delta \sigma \ra = \int_{\bd M} \sum_{i=1}^2  k_i \s_i \Delta \s_i
 =\sum_{i=1}^2  k_i \| \grad \s_i \|^2 
 = \sum_{i,j=1}^2  k_i \| \nabla_j \s_i \|^2 .$$

 Combining the last two equations and using $k_1 k_2 = 1$ we obtain
 \begin{eqnarray*}
 && 4 \la S\sigma, \Delta \sigma \ra - H ( \| \delta S \sigma \|^2 + \| d\sigma \|^2  ) \\
 &&= (3- k_1^2) k_1\| \nabla_1 \s_1 \|^2 
 + (3 k_1 - k_2) \| \nabla_2 \s_1 \|^2
 + (3 k_2 - k_1) \| \nabla_1 \s_2 \|^2
 +  (3- k_2^2) k_2 \| \nabla_2 \s_2 \|^2 \\
 &&= \sum_{i,j=1}^2 (3 - k_i^2) k_j \| \nabla_i \s_j \|^2
 \end{eqnarray*}
 
To examine the other terms, write $\delta d \sigma = a_1 \theta_1 +a_2 \theta_2$.
Then $$\hat S \delta d \sigma = k_2 a_1 \theta_1 +k_1 a_2 \theta_2$$ since $\hat S = S^{-1}$
and $k_1 k_2 = 1$.  Hence
\begin{eqnarray*}
\frac{\varepsilon}{2} \la \delta d \sigma, \delta d  \sigma \ra 
  -\left(\frac{\varepsilon}{2}\right)^2 \la \delta d \sigma, \hat S \delta d \sigma \ra
&=& \frac{\varepsilon}{2} \int_{\bd M} \left( (a_1^2+a_2^2)
- \frac{\varepsilon}{2}(k_2 a_1^2 + k_1 a_2^2)  
\right) \, dA \\
&&= \frac{\varepsilon}{2}  \int_{\bd M} 
\left( (1 -  \frac{\varepsilon}{2} k_2)a_1^2 +
(1 -  \frac{\varepsilon}{2} k_1 )a_2^2 \right) \, dA .
\end{eqnarray*}
This will be non-negative provided $\displaystyle\frac{\varepsilon}{2} \le \frac{1}{k_2}=k_1$ and
$\displaystyle\frac{\varepsilon}{2} \le \frac{1}{k_1}=k_2$, that is,  if
$\displaystyle 0\le \frac{\varepsilon}{2} \le \min(k_1,k_2)$.

\medskip
 This gives our final conclusion:
 \begin{theorem} \label{positiveb}Let $M$ be a hyperbolic 3-manifold with tubular
 boundary and let $A=S+\frac{\varepsilon}{2}\delta d$ where $\varepsilon>0$
 is a constant.  
 If  $2 \tilde h = \curl v = *d \sigma$,  $2h = - \curl \tilde v = - * d \tilde \sigma$
and $\tilde \sigma = *A\sigma$,  then
the  Weitzenb\"ock correction term $b = \int_{\bd M} \eta_c \wedge \tilde \eta_c$ is
\begin{eqnarray*} 
\frac{1}{4}  \sum_{i,j=1}^2 (3 - k_i^2) k_j \| \nabla_i \s_j \|^2
 +  \frac{\varepsilon}{2}  \int_{\bd M} 
\left( (k_2 -  \frac{\varepsilon}{2})a_1^2 +
(k_1 -  \frac{\varepsilon}{2})a_2^2 \right) \, dA .
\end{eqnarray*}
 Hence the boundary term is non-negative
 if the principal curvatures $k_1, k_2$
 satisfy $$\frac{1}{\sqrt{3}} \le k_1 \le k_2 \le \sqrt{3} $$
 and $\varepsilon \le 2 k_1$.
 \end{theorem}

\section{Boundary Values}\label{bvalues}

In this section we will describe a boundary value problem that will allow us to
find harmonic representatives for infinitesimal deformations of hyperbolic
$3$-manifolds with tubular boundary whose boundary values are of the
form discussed in the previous section.  This will allow us to make
statements about the boundary term in the \W formula which, in turn, will
lead to local rigidity results for such $3$-manifolds.  Those results and
other applications will be discussed in Section 5.

In Section 2 we saw that finding a harmonic representative $\omega$ for an infinitesimal
deformation amounts to finding a {\it real-valued} $1$-form $\tau$ which is a solution to the
equation $(\hat \lap + 4) \tau = \zeta$.  Here, $\zeta$ is a smooth, real-valued $1$-form
which equals zero in a neighborhood of the boundary and $\hat \lap$ is the usual Laplacian
on real-valued 1-forms on $\N$.  The $1$-form $\tau$ is dual
to a vector field $u$ on $\N$ which is the real part of an $E$-valued section $s$ and
the coboundary $d_E\, s$ is added to the original $E$-valued $1$-form in order to
make it globally harmonic.  The boundary behavior of $\tau$ determines that of
$s$ and hence of $d_E\, s,$ providing information about the boundary values of $\omega$.

In order to have any control over the behavior of $\tau$ near the boundary, it is necessary
to put restrictions on the domain of the operator $(\hat \lap + 4)$.  However, the
restrictions must still allow the
above problem to be solvable.  Below, we will define boundary data that the
real-valued $1$-form $\tau$ must satisfy which make this operator elliptic,
self-adjoint with trivial kernel.  Standard theory (Chapter X in \cite{Horm})
then implies that the above
problem is uniquely solvable; when $\zeta$ is smooth, as it is in our situation,
the solution $\tau$ will be smooth.

There are many choices for such boundary conditions.  Standard examples include prescribing
that either the value or the normal derivative of $\tau$ be zero, analogous to Dirichlet
and Neumann conditions for the Laplacian on real-valued functions.  However, our
choice is motivated by the further
condition that the resulting \W correction term $b$ defined in (\ref{bdef}) be positive.
None of the more standard choices of boundary data have this property.

In order to describe our boundary conditions we first need to establish some notation.

The  above Laplacian on $1$-forms, $\hat \lap$,  equals
$$ \hat d\,\hat \delta + \hat \delta\,\hat d $$
where $\hat d$ is exterior differentiation on $\N$ and $\hat \delta$ its adjoint.
We will denote by $d$ and $\delta$ the corresponding operators on $\bd \N$.
Similarly, we use the notation $\hat *$ to denote the $3$-dimensional Hodge star operator on
forms, reserving the notation $*$ for the corresponding operator on the boundary.

It is also useful to define operators $d_S = Sd$ and $\delta_S = \delta S$ operating on
functions on the boundary and on tangential 1-forms respectively. Here $S$ is the 2nd
fundamental form or shape operator on the boundary, with normal chosen so that, in our
situation with concave boundary, $S$ is positive definite.
We also define $\Delta_S = \delta d + d_S \delta_S$, which acts on tangential 1-forms.

Recall the basic setup from Section 2:  We begin with an $E$-valued $1$-form
$\hat \omega$ which represents the cohomology class in $H^1(\N;E)$ determined
by our infinitesimal deformation.  It satisfies $d_E \hat \omega = 0$, but,
in general, $\delta_E \hat \omega \neq 0$.  To find a harmonic representative
we must find a globally defined $E$-valued section $s$ satisfying the equation
$\delta_E d_E s ~ = ~ - \delta_E \hat \omega.$
Then $\omega = \hat \omega + d_E s$ is a harmonic
representative in the same cohomology class as $\hat \omega$.

Decomposing $s$ into its real and imaginary parts, we write
$s = u + i \,\tilde u$ where $u$ and $\tilde u$ can be viewed as vector
fields on $\N$.  We can assume that $2 \tilde u = \curl u$ (by choosing $s$ to be a
canonical lift, see Section 2 of \cite{HK1}.  As discussed in Section 3
of the current paper, this is equivalent to the real part of $d_E s$ being symmetric.)
Thus, it suffices to find $u$.  
This is equivalent to solving the equation
$(\hat \lap + 4) \tau = \zeta$ where $\zeta$ is the $1$-form dual to the
real part of the $E$-valued section $- \delta_E \hat \omega$ and $\tau$
is the $1$-form dual to $u$.  The equation $2 \tilde u = \curl u$
is equivalent to the equation $2 \tilde \tau = \hat * \hat d \tau$,
where $\tilde \tau$ is the $1$-form dual to $\tilde u$.

As we have done before, we can decompose 1-forms on (a neighborhood of) the boundary into their normal
and tangential parts.  In particular we write
\begin{equation} \label{taudecomp}
\tau = h \,dr + \sigma  \qquad{\rm ~and~} \qquad 2 \tilde \tau = \hat * \hat d \tau = 2(\tilde h \,dr + \tilde \sigma),
\end{equation}
where $dr$ denotes the 1-form dual to the inward pointing unit normal
and $\sigma, \tilde \sigma$ are tangential 1-forms.

We now describe a $1$-parameter family of boundary conditions, parametrized by a
parameter $\varepsilon$.  It is assumed that $\varepsilon > 0$ and is a constant.
Using the notation established above, the boundary conditions can be expressed as:
\begin{equation} \label{bdcond1}
\hat \delta \tau - 2 \varepsilon (\delta_S \sigma -2h) = 0
\end{equation}
\begin{equation} \label{bdcond2}
2  \tilde \sigma - *(2 S \sigma + \varepsilon (\lap_S \sigma - 2 d_S h)) = 0.
\end{equation}

At the end of this section we will show that the boundary value problem of solving
$(\hat \lap +4) \tau = \zeta$ subject to these boundary conditions is
elliptic and that the operator is positive, self-adjoint.  This implies that
there will be a unique solution and that the solution will be smooth on the
entire manifold with boundary.

However, in order to provide some motivation for choosing these fairly complicated
boundary conditions, we will first assume the existence of such a solution and analyze the properties of the harmonic $E$-valued $1$-form $\omega$ that we obtain
from $\tau$.

Note that $\omega$ is $d_E$ closed and, hence, is the image under
$d_E$ of a {\em locally}
defined section of $E$ whose real part is a locally defined vector field.
The divergence of this vector field is just the trace of
the real part of $\omega$, viewed as an element of ${\rm Hom}\,(TM,TM)$.
Thus, although the vector field is only locally defined,
its divergence is a globally well-defined function.  By abuse of language
we will refer to this as the ``divergence of $\omega$".

The main step is to show that when $\tau$ satisfies the above
boundary conditions, the resulting harmonic $E$-valued $1$-form $\omega$
has divergence identically zero.
This will imply that the stronger harmonicity equations (\ref{etaeqn}) and 
(\ref{omegadecomp}) hold and, thus, that the results from Section 2 and the
computations from Section 3 all apply.  It will also show that $\tau$
in fact satisfies boundary conditions that are stronger and simpler than
(\ref{bdcond1}) and (\ref{bdcond2}).

\begin{proposition} \label{divfree} Let $\omega = \hat \omega + d_E s$ be a
harmonic $E$-valued $1$-form on a compact hyperbolic $3$-manifold with
tubular boundary, where $\hat \omega$ is in standard form near the
boundary and $s = u + i \,\tilde u$ is a global section of $E$.
Let $\tau, \tilde \tau$ be the $1$-forms dual to the vector fields
$u, \tilde u$, respectively.  If $\tau$ and $\tilde \tau$ are
decomposed as in (\ref{taudecomp}) and satisfy the boundary conditions
(\ref{bdcond1}) and (\ref{bdcond2}), then the divergence of
$\omega$ is identically zero.
\end{proposition}

\begin{proof}

We denote by $\tr$ the divergence of the harmonic deformation $\omega$.
Harmonicity of $\omega$ implies (see \cite[Lemma 2.4)]{HK1}  that
$$ (\hat \delta \hat d  + 4)\tr  = 0.$$
This equation holds on all of $\N$ and, taking the $L^2$ dot product on $\N$ of $\tr$
with this equation, we conclude that $\langle (\hat \delta \hat d  + 4)\tr , \tr \rangle = 0.$ Integrating by parts gives
$$\langle \hat d \tr, \hat d \tr \rangle + 4 \langle \tr, \tr \rangle +
\int_\bdM~   \tr \wedge \hat * \hat d \tr ~=~ 0$$
where the boundary is oriented using the {\it inward} normal.
If we show that the boundary integral
$$\int_\bdM~   \tr \wedge \hat * \hat d \tr$$
is non-negative,
it will follow that $\tr= 0$
and $\hat d \tr = 0$ on all of $\N$. In particular, we will have shown
that $\omega$ is divergence-free.

Since $\hat \omega$ equals some standard harmonic $E$-valued $1$-form in a 
neighborhood of the boundary, and since standard forms are all divergence-free,
the divergence of $\omega = \hat \omega + d_E s$ just equals the divergence of 
$d_E s$ in a neighborhood of the boundary.  
By definition the latter equals the 
divergence of the vector field $u$ which is the real part of $s$;
this equals $- \hat \delta \tau$, since $\tau$ is the $1$-form dual to $u$.

Since $\tr = - \hat \delta \tau$ in a neighborhood of the boundary, we can
use the boundary condition (\ref{bdcond1}) on $\tau$
when computing the boundary integral. Also, since $\hat \omega$ is harmonic in
a neighborhood of the boundary, $\delta_E \hat \omega = 0$ in a neighborhood
of the boundary which in turn implies that $\delta_E d_E s$  is zero near the boundary.  As discussed in Section 2, this means that
$(\hat \lap + 4) \tau = 0$ near the boundary.

The second term in the integrand becomes $- \hat * \hat d \hat \delta \tau$, but,
because $\tau$ satisfies $(\hat d \hat \delta + \hat \delta \hat d + 4) \tau = 0,$
this equals $\hat *(\hat \delta \hat d + 4) \tau$.  Since the integral is over the
boundary, only the tangential part of the integrand appears.  The tangential
part of $\hat*\hat \delta \hat d \tau = \hat d \hat * \hat d \tau$ equals
$d (2 \tilde \sigma)$ where  $2 \tilde \sigma$ is the tangential
part of $\hat * \hat d \tau$ and $d$ is the exterior derivative on the boundary.
The tangential part of $\hat * \tau$ equals $h\, dA$, where $dA$ is the area form on
the boundary. Rewriting $2 d \tilde \sigma$ as $ 2 \delta (* \tilde \sigma) ~ dA$,
the boundary integral can then be written as
$$\int_\bdM~  (-\hat \delta \tau)  ( 4 h + \delta (*2  \tilde \sigma)) ~ dA.$$

Using (\ref{bdcond1}) and (\ref{bdcond2}), the boundary integral becomes
$$\int_\bdM ~  2 \varepsilon (2h -\delta_S \sigma)  ( 4 h -
\delta (2 S \sigma + \varepsilon (\lap_S \sigma - 2 d_S h)) ~ dA.$$

This simplifies to
$$ 2 \varepsilon ~ \int _\bdM ~(2h - \delta_S \sigma) ~ ((2  +  \varepsilon \delta d_S) (2h - \delta_S \sigma)) ~ dA,$$
which equals
$$2 \varepsilon ~ \langle (2h - \delta_S \sigma), (2  +  \varepsilon \delta d_S) (2h - \delta_S \sigma) \rangle_{\bdM},$$
where $\langle \cdot, \cdot \rangle_{\bdM}$
denotes the $L^2$ dot product of $1$-forms on $\bdM$.

We see that  $2 + \varepsilon \delta d_S = 2 + \varepsilon \delta S^{\frac 1 2} S^{\frac 1 2} d$
is a positive operator since $S$ is positive self-adjoint and hence has  a positive
self-adjoint square root.  Thus the boundary integral is non-negative and
we conclude that the deformation is divergence free on $\N$ as desired.
\end{proof}

The fact that $\omega$ is divergence-free provides us with further
information about the boundary behavior of the solution $\tau$.
In particular, since $\tr = -\hat \delta \tau$
near the boundary, one concludes, as in the previous proof, that
$\hat \delta \tau = 0$ near the boundary.  Together with the first
boundary condition (\ref{bdcond1}), this implies that
$$2\, h = \delta_S \sigma,$$ 
and the second boundary condition (\ref{bdcond2}) simplifies to
$$ 2  \tilde \sigma =  *(2 S \sigma + \varepsilon \delta d \sigma).$$

The computations of the general Weitzenb\"ock  correction
term in the previous section were all expressed in terms of a
general differential operator $A$ on tangential
$1$-forms.  The operator $A$ expresses the relation between the two
tangential $1$-forms $\sigma, \tilde \sigma$ and is determined
by the equation (\ref{Adef}) $\tilde \sigma = * A \sigma$.  With
our choice of boundary conditions (\ref{bdcond1}) and (\ref{bdcond2}),
we see that
$$ A \sigma = S \sigma + \frac 12 \varepsilon \delta d \sigma$$
which is precisely the value (\ref{Avalue}) for the operator $A$
that we wished to obtain.  It was with this result in mind that we were
led to our boundary conditions.

We noted in the previous section that in order for the Weitzenb\"ock  
correction term $b$ in
(\ref{bwithA}) to be positive, it is necessary for the $0$-order term of $A \sigma$ to
equal $S \sigma$ and that choosing $A \sigma = S \sigma$ did give a positive value
for $b$. Indeed, a natural choice for our boundary conditions would have been to
set $\varepsilon = 0$ in (\ref{bdcond1}) and (\ref{bdcond2});  this would have led
to the value $A \sigma = S \sigma$.  However, those conditions do not lead to
an elliptic boundary value problem and it was necessary to perturb this
natural choice to obtain an elliptic problem.  To do so in such a way
that the resulting harmonic deformation was divergence-free and so
that $b$ was still positive required some delicacy
and led to the more complicated form of the boundary conditions.

\bigskip
We are now in a position to apply the conclusions from the previous sections.

\bigskip

Let $\omega = \hat \omega + d_E s$ equal the harmonic $E$-valued
$1$-form obtained from our boundary value problem, and decompose $\omega$
into its real and imaginary parts as
$\eta + i \tilde \eta$ where $\eta, \tilde \eta$ are elements of
${\rm Hom}\,(TM,TM)$.  We have assumed that $\eta$ is symmetric and 
Theorem \ref{divfree} implies that it is traceless.
Therefore the stronger harmonicity equation (\ref{etaeqn}) holds.  As discussed in
Section 2, this implies that $\tilde \eta$ is also traceless and symmetric
and equals $*D\eta$.  This, in turn, allows to conclude that the
\W formula (\ref{bdryeq}) holds.

In a neighborhood of the boundary, we can write $\omega = \omega_0 + \omega_c$,
where $\omega_0$ is a standard form and $\omega_c = d_E s$.  Since the real
and imaginary parts of $\omega_0$ are also symmetric and traceless, the
same will be true for $\omega_c$.  Writing $\omega_c = d_E s = \eta_c + i \tilde \eta_c,$
we see that the hypotheses of Theorem \ref{harmonicweitzbdry} hold for $\eta_c, \tilde \eta_c.$
Since we have also concluded that, for any $\omega$ obtained using our boundary conditions,
the operator $A$ satisfies (\ref{Avalue}), Theorem \ref{positiveb} applies.
If we write  $\omega_0 = \eta_0 + i \tilde \eta_0$,
then, as discussed in Section 2, $\tilde \eta = *D\eta$ and $\tilde \eta_0 = *D\eta_0.$
It follows that $\tilde \eta_c = *D\eta_c$ as well. Hence, the boundary integral
$b$ in (\ref{bdef}) equals $-b_R(\eta_c, \eta_c)$ where $b_R(\eta_c, \eta_c)$ is the
contribution from the correction term $\omega_c$ to the \W formula (\ref{bdryeq}):
$$ b = \int_{\bd M} \eta_c \wedge \tilde\eta_c = \int_{\bd M} \eta_c \wedge *D\eta_c =
- b_R(\eta_c,\eta_c)$$

Thus, assuming the existence of the solution $\tau$ of our boundary value
problem (proved in Theorem \ref{existence} below), Theorem \ref{positiveb}
implies:

\begin{theorem} \label{negativebR} Let $\N$ be a hyperbolic $3$-manifold with
tubular boundary whose principal curvatures $k_1, k_2$
satisfy $\frac{1}{\sqrt{3}} \le k_1 \le k_2 \le \sqrt{3}. $
Then, for any infinitesimal deformation, there is a harmonic representative
$\omega$ so that the correction term $b_R(\eta_c,\eta_c)$ is non-positive.
\end{theorem}

Recall from Section 2 that, for any divergence-free harmonic $\omega$,
we have the \W formula (\ref{bdryeq}, \ref{bdrydecomp}):
\begin{eqnarray*}||D\eta||^2 ~+~ ||\eta||^2 ~=~ \b_R(\eta,\eta) = \b_R(\eta_0,\eta_0)
+ \b_R(\eta_c,\eta_c).
\end{eqnarray*}
We immediately obtain the following corollary which will be crucial in
proving the applications in the next section.  Note that the standard form
$\omega_0$ depends only on the infinitesimal variation of the holonomy
of the boundary; thus, it and its real part $\eta_0$ are invariants of
the cohomology class of the infinitesimal deformation, independent of the
choice of representative.

\begin{corollary} \label{positivestandard} Let $\N$ be a hyperbolic $3$-manifold with
tubular boundary whose principal curvatures $k_1, k_2$
satisfy $\frac{1}{\sqrt{3}} \le k_1 \le k_2 \le \sqrt{3}. $
Then, for any non-trivial infinitesimal deformation of $\N$ we obtain $b_R(\eta_0,\eta_0) > 0.$
\end{corollary}

Finally, we must justify our claim that we can always solve our given
boundary value problem.

For our purposes a {\it differential operator} on a manifold $\N$ with boundary
consists of a differential operator $P$ from
$C^\infty(F)$, the $C^{\infty}$ sections of a bundle $F$ over $M$ to $C^\infty(G)$, where $G$ is
another such bundle, together with
a collection $\{p_1,...,p_r\}$ of differential operators from $C^\infty(F)$ to $\oplus_i C^\infty(G_i)$ where
$\oplus_i G_i$ is a direct sum decomposition of the bundle $G$, restricted to the boundary.
Such an operator will be denoted by $(P; \{p_1,...,p_r\}).$
In our case, the bundles $F$ and $G$ are both equal to the bundle of $1$-forms on $\N$
and we decompose this bundle on the boundary as the direct sum of its normal and tangential
parts.  The main operator $P$ is $\hat \lap + 4$ and $p_1, p_2$ equal the operators
on the left-hand sides of (\ref{bdcond1}) and (\ref{bdcond2}), respectively.

The remainder of this section will be devoted to proving the following theorem.

\begin{theorem} \label{existence} Let $\N$ be a compact hyperbolic $3$-manifold with
tubular boundary.  Then, for any constant $\varepsilon > 0$, the differential operator
$(P;\{p_1,p_2\})$ on the bundle of $C^\infty$ real-valued $1$-forms on $\N$ defined, using the notation above, by
\begin{eqnarray*}
P(\tau) &=& (\hat \lap +4) (\tau)\\
p_1(\tau) &=& \hat \delta (\tau) - 2 \varepsilon (\delta_S \sigma -2h)\\
p_2(\tau) &=&  2 \tilde \sigma - *(2 S \sigma + \varepsilon (\lap_S \sigma - 2 d_S h))
\end{eqnarray*}
is elliptic.  On the subspace where $p_1(\tau) = p_2(\tau) = 0$ it is positive 
and self-adjoint.  In particular, for any smooth $1$-form
$\zeta$, there is a unique solution to $P(\tau) = \zeta, p_1(\tau)=p_2(\tau)=0$
and that solution is smooth on all of $\N$.
\end{theorem}

\begin{proof}

We will first show that this operator is self-adjoint and positive.

Recall that $\hat\lap = \hat d \hat \delta + \hat \delta \hat d$.  Then
for any real-valued $1$-forms $\tau, \psi$ on $\N$, integration by parts gives us:
$$\langle \hat \lap \tau, \psi\rangle = \langle \hat d \tau, \hat d \psi \rangle ~+~
\langle \hat \delta \tau, \hat \delta \psi \rangle + B(\tau,\psi),$$
where $\langle \cdot, \cdot \rangle$ is the $L^2$ inner product on $\N$ and
$B(\tau,\psi)$ is a boundary term which is given by an integral over
the boundary.

The operator $\hat \lap$ (hence $\hat \lap + 4$) is self-adjoint as long as
$$\langle \hat \lap \tau, \psi\rangle ~-~ \langle\hat \lap \psi, \tau\rangle
~=~ B(\tau,\psi)\,-\, B(\psi, \tau) ~=~ 0.$$
The operator $\hat \lap + 4$ will have trivial kernel as long as
$$\langle (\hat \lap +4) \tau, \tau \rangle > 0$$
 for any non-zero $\tau$.  Letting $\tau = \psi$ above this will be guaranteed
as long as we have
$$B(\tau,\tau) \geq 0.$$

\medskip

Using Green's identity, we obtain the following formula for the boundary term, where
the boundary is oriented with respect to the {\it inward} normal:
$$ B(\tau, \psi) ~=~ -( \int_\bdM ~ \hat *\hat d\tau \wedge \psi ~+~ \hat \delta \tau \wedge\hat * \psi).$$
Again we use the notation $\hat *$ to denote the $3$-dimensional star operator on
forms, reserving the notation $*$ for the corresponding operator on the boundary.
\medskip

As before we decompose $\tau$ as $\tau = h \, dr + \sigma$ and let
$ 2 \tilde \sigma$ equal the tangential part of $\hat * \hat d \tau$.
If we decompose $\psi$ as $\psi = k \,dr + \phi$, we can write
$$B(\tau, \psi) ~=~ \int_\bdM ~ 2 \phi \wedge \tilde \sigma  ~-~ \hat \delta \tau \wedge *k.$$
Using the boundary conditions (\ref{bdcond1}) and (\ref{bdcond2}),
the boundary term becomes
\begin{eqnarray*}
B(\tau, \psi) &=& \int_\bdM ~  \phi \wedge *(2 S \sigma + \varepsilon (\lap_S \sigma - 2 d_S h)
 ~+~ 2 \varepsilon  (2h - \delta_S \sigma ) \wedge *k) \\
&=& 2 \langle \phi, S \sigma \rangle + \varepsilon  (\langle \phi, \delta d \sigma  +
d_S (\delta_S \sigma - 2 h) \rangle
+ \langle  -2k, (\delta_S \sigma - 2 h) \rangle)\\
&=& 2 \langle \phi, S \sigma \rangle + \varepsilon  (\langle d \phi,  d \sigma \rangle +
\langle (\delta_S \phi - 2k),  (\delta_S \sigma - 2 h) \rangle)
\end{eqnarray*}
where $\langle \cdot, \cdot \rangle$ denotes the $L^2$ inner product on
forms on the boundary oriented by the {\em inward} normal and we have used
the definition $\Delta_S = \delta d + d_S \delta_S$.

\medskip

It is apparent from this formula and the fact that
$S$ is symmetric and positive definite, that the boundary term is
symmetric in $\tau$ and $\psi$ and non-negative when $\psi = \tau$.
It follows that the operator $\lap + 4$ is self-adjoint and positive
definite with these boundary conditions.

\bigskip

Finally, we must show that this boundary value problem is elliptic.
To see that the boundary conditions lead to an
elliptic boundary value problem, we consider each of the operators,
$P, p_1, p_2$ where the ranges of the two latter operators are
the sub-bundles of normal and tangential parts of $1$-forms on the boundary.
We then take the top order terms of each of these operators.  It is a subtlety
of differential operators on bundles
that ellipticity may depend on the choice of decomposition of the target
bundle, since this affects what the top order terms are.

To show that the system is elliptic one considers the symbols of the operators.
This amounts to looking at the system in local coordinates, fixing the coefficients
of the operators by evaluating at a boundary point, and taking only the top
order terms in each operator.  One then considers the homogeneous, constant
coefficient problem in the upper half space of $\R^n$ given by these simplified
operators.  We refer the reader
to \cite{Horm}, Chapter X or \cite{atiyah}, Appendix I for a full description.
Below we will see how the process works in our specific case.

For $\hat \lap + 4$ taking the symbol simply gives the standard laplacian in $\R^3$
which is well-known to be elliptic.
Since $2 \tilde \sigma$ equals the tangential part of $\hat * \hat d \tau,$
it is obtained by applying a first order operator to $\tau$.  One then easily
sees that $p_1$ is
of first order and $p_2$ is of order 2.  Taking the top order terms,
the boundary operators simplify to
\begin{eqnarray*} &&\hat \delta \tau - 2 \varepsilon (\delta_S \sigma), \\
&& \varepsilon \lap_S \sigma.
 \end{eqnarray*}

These operators are still defined in terms of the hyperbolic metric.  It is easy
to check that, taking natural orthonormal coordinates at any point on
the boundary torus the coefficients of the top order terms are independent
of the point chosen.  The operators simply become the same operators
viewed in the upper half-space
$\R^3_{+} = \{(x,y,t) | t \geq 0\}$ with the standard Euclidean metric.
Multiplication by $S$ becomes multiplication in the $(x,y)$-plane by
the diagonal matrix with diagonal entries $k=k_1, k^{-1}=k_2$ which
are the principal curvatures of the boundary torus.
We are now left with the simplified system of solving
$\hat \lap \tau = 0$ in the upper half-space with homogeneous boundary
conditions determined by these simplified boundary operators.
By definition,
the original system is elliptic if and only if
this simplified  system has no non-trivial bounded solutions.
It suffices to show that there are no non-trivial bounded
solutions $\tau(x,y,t)$ of the form
$$\tau (t) e^{i(\zeta \cdot (x,y))},$$
where $\zeta = (a,b)$ is any
non-zero vector in the boundary plane.  The solutions, $\tau(t)$,  to $\hat \lap \tau = 0$
for a given choice
of $\zeta$ are linear combinations of $e^{|\zeta| \, t}$ and $e^{-|\zeta| \,t}$.  Since we
are only interested in bounded solutions, only scalar multiples of the latter function
will appear.  In particular, we have that $ \tau^{\prime}(t) = -|\zeta| \tau(t)$.

We decompose $\tau$ into its normal and tangential components which
we again denote by $h$ and $\sigma$, respectively.  Viewing $\tau$ as a $3$-dimensional
vector field and $\sigma$ as a $2$-dimensional one, $-\hat \delta \tau$ is the
divergence of $\tau$ which equals $h^{\prime}$ plus the divergence of $\sigma$.
Similarly, $-\delta_S \sigma$ is the divergence of $S \sigma$.  Note that
all calculations are done with respect to the Euclidean metric.

For solutions of this form with $\zeta$ fixed, the boundary conditions become
\begin{eqnarray*}
h(0) |\zeta| - i\, \zeta \cdot \sigma(0) + i \, 2\varepsilon \zeta \cdot S \sigma(0) &=& 0 \\
 \varepsilon L_S (\zeta) \sigma(0) &=& 0,
 \end{eqnarray*}
where $L_S(\zeta)$ is the matrix computed below.

We'll see that  $L_S(\zeta)$ is invertible, so the second boundary condition
implies that $\sigma(0) = 0,$
and thus, the first implies that $h(0) = 0$.  This means that
$\tau(0) = 0$ and, since $ \tau^{\prime}(0) = -|\zeta| \tau(0)$,  we conclude that any
solution must be trivial.

To compute the matrix in the second boundary condition, we view $\sigma$ as a
$1$-form on the boundary of the upper half-space.  Recall that $\Delta_S =
\delta d + d_S \delta_S = \delta d + S d \delta S$.  It's a standard calculation
that for $\zeta = (a,b)$, the symbols for $\delta d$ and $d \delta$ are, respectively,
multiplication by the matrices
$$ \displaystyle{
 \begin {bmatrix} b^2& -ab\\
 -ab& a^2
 \end{bmatrix}  ~~~~~~~~
 ~~~~{\rm and}
 \begin {bmatrix} a^2& ab\\
 ab& b^2
 \end{bmatrix}}   $$

The matrix $S$ is diagonal with entries $k$, $k^{-1}$ on
the diagonal.  It follows that the matrix corresponding to the operator $\lap_S$
is
$$ \displaystyle{ L_S(\zeta) ~~=~~
 \begin {bmatrix} k ^2 a^2 + b^2& 0\\
0 & a^2 + k^{-2} b^2
 \end{bmatrix}} $$
which is clearly invertible for all $(a,b) \neq 0.$
\end{proof}

\bigskip

It is worth pointing out that, when $\varepsilon = 0$ in the original boundary conditions,
the top order term in the second equation is $2 \tilde \sigma$ which is of first order and
the simplified conditions become
\begin{eqnarray*}
\hat \delta \tau &=& 0 \\
2 \tilde \sigma  &=& 0.
\end{eqnarray*}

An easy calculation shows that the first order part of $2 \tilde \sigma$ equals  $\sigma^{\prime} - dh$,
so that the corresponding linear equations used to determine ellipticity are
\begin{eqnarray*}
h(0) |\zeta| - i\, \zeta \cdot \sigma(0) &=& 0 \\
\sigma(0) |\zeta| + i\, h(0) \zeta  &=& 0.
\end{eqnarray*}

This system is seen to have a non-trivial solution given by
$ i \, \sigma(0) = h(0) \frac {\zeta} {|\zeta|}$ for any $\zeta \neq 0$,
and thus the system is not elliptic.  This is the reason we needed to perturb the
system by adding a small second order term in the second boundary condition.
The term added to the first boundary condition was necessary to keep the system
self adjoint.  The precise form of these added terms was determined by other conditions
necessary to conclude that the system was divergence free and led to a positive
value for the \W boundary correction $b$.

\section{Applications to hyperbolic Dehn Surgery Space}\label{applic}

We now apply our harmonic deformation theory to study generalized hyperbolic Dehn surgery as introduced by Thurston in \cite{thnotes}.  

\subsection{Geometry of tubular boundaries}\label{geom_tub}
\hfill\\
Let $M$ be a compact orientable hyperbolic 3-manifold $M$ with tubular boundary,
and let $T=T_R$ be a torus in $\bd M$ of tube radius $R <\infty$. Then $T$ has
principal curvatures
$k_1=\coth R$ and $k_2 =\tanh R$, and the intrinsic Euclidean metric on $T$
has the form $\E^2/ \Gamma$ where $\Gamma \cong \Z^2$ acts as translations of $\E^2$, and the holonomy of the Euclidean structure gives an isomorphism $h : \pi_1(T) \to \Gamma$.

The developing map for $M$ restricts to 
an isometric immersion 
$\Phi : \tilde T = \E^2 \to \H^3$ taking the universal cover $\tilde T$ of $T$ to the surface of a cylinder of radius $R$ in $\H^3$; $\Phi$ is uniquely defined up to composition with isometries of $\H^3$. 

Explicitly, if we 
choose standard Cartesian coordinates on $\E^2$ with $x_1,x_2$ coordinate axes in the directions of the principal curvatures $k_1, k_2$ respectively, we can take $\Phi(x_1,x_2)$ to be the point with hyperbolic cylindrical coordinates 
$\displaystyle \theta = {x_1 \over \sinh R}$ and
$\displaystyle \zeta = {x_2 \over \cosh R}$. 
Then for each $\gamma \in \pi_1(T)$, 
the {\em complex length} of $\gamma$ is given by
\begin{eqnarray} \cl (\gamma) = {x_2 \over \cosh R} +  i {x_1 \over \sinh R},
\label{C_len_def}
\end{eqnarray}
where the Euclidean holonomy 
$h(\gamma)$ is a translation with components $(x_1,x_2)$.
Once we choose an orientation on $T$ the complex length is uniquely defined up to sign; changing sign corresponds
to composing $\Phi$ with a 180 degree rotation of $\H^3$ taking the cylinder to itself.
In the limiting case as 
$R\to\infty$, we obtain a horospherical torus;  then we define 
$\cl (\gamma) = 0$ for all $\gamma \in \pi_1(T)$.

The {\em Euclidean length} $L = L(\gamma)$ of a closed curve $\gamma$ on the Euclidean torus $T$ is the length of a geodesic in its homotopy class. This is the length of the translation $h(\gamma)= (x_1,x_2)$, so we have
$L^2 = x_1^2 + x_2^2$, and this can be written terms of the complex length as follows:
\begin{eqnarray}L^2 = (\cosh R \, \Re \cl)^2 + (\sinh R \, \Im \cl)^2.
\label{L_vs_cL}
\end{eqnarray}

Next we extend some geometric notions from simple closed curves on $T=T_R$
to arbitrary homology classes in $H_1(T;\R)$.

The {\em complex length} of closed curves on $T$ gives a $\Z$-linear function
$\cl : H_1(T;\Z)\cong \pi_1(T)  \to \C$. We define the complex length $\cl$ for each 
element of $H_1(T;\R)$ by extending this to an $\R$-linear 
function $\cl : H_1(T;\R) \to \C$. 

We can regard the {\em generalized Dehn surgery coefficient} on $T$ as the 
homology class  $c=\cl^{-1}(2\pi i)$ in $H_1(T;\R)$ whenever $\cl$ is invertible.
(Note that (\ref{C_len_def}) implies that $\cl$ is invertible whenever $T$ has tube radius $R<\infty$.) 
After choosing a basis $a,b$ for $H_1(T;\Z) \cong \Z^2$ this corresponds to the
element $(x,y)$ in $\R^2$ such that $x\cl(a)+y\cl(b) = 2 \pi i$, giving the generalized Dehn surgery coefficient as defined in \cite{thnotes}.  

If $\bd M$ consists of $k$ tori $T_1$, \ldots $T_k$, we have a complex length
function $\cl_j : H_1(T_j;\R) \to \C$ for each $j$; the direct sum of these gives
a function $\cl : H_1(\bd M;\R) = \oplus_j   H_1(T_j;\R) \to \C^k$.
We then define the generalized Dehn surgery coefficient
to be the homology class $c=\cl^{-1}(2\pi i, \ldots, 2 \pi i) \in H_1(\bd M;\R) = \oplus_j   H_1(T_j;\R)$ whenever $\cl$ is invertible.

The {\em Euclidean length} $L$ of closed curves on $T$ gives
a quadratic form $L^2$ on $H_1(T;\Z)$, and this extends naturally to a positive definite quadratic form on $H_1(T;\R)$.
If we choose
a basis $a,b$ for $H_1(T;\Z)$ and let $c= p a + q b$ where $p,q \in \Z$,
then we can write $L(p a+q b)^2 =A p^2 +2B p q + C q^2$ where $A,B,C$ are constants. We use the same formula
to {\em define} the length $L(pa+qb)$ whenever $p,q \in \R$. Then the
relationship (\ref{L_vs_cL}) between Euclidean length and complex length continues to hold for all homology classes in $H_1(T;\R)$.

Next we discuss some other geometric quantities that will be important
in our arguments.

The {\em area} of the torus $T=T_R$ can also be expressed in terms of complex lengths.
If $a,b$ is a basis for $H_1(T;\Z)$, then the area of $T$ is just the area of the parallelogram with sides given by the Euclidean translations $h(a), h(b)$.
If  $ l_a +  i \theta_a$ and $l_b +  i \theta_b$
denote the complex lengths of $a$ and $b$, then
$$h(a) = (\theta_a \sinh R,  l_a \cosh R) {\rm ~and~}  h(b) = (\theta_b \sinh R,  l_b \cosh R)$$
so 
$$\area(T_R) = \begin{vmatrix}\theta_a \sinh R  &  l_a \cosh R\\
\theta_b \sinh R  &  l_b \cosh R\end{vmatrix}
= \sinh R \cosh R \, (l_b \theta_a  - l_a \theta_b),$$
provided $a,b$ are oriented so the above determinant is positive.
(Note that this is the same for any basis
$a,b$ which is oriented compatibly with $T$.)

We now define the {\em visual area} of the torus $T_R$ to be
\begin{eqnarray}\ca  = {\area(T_R) \over \sinh R \cosh R}.
\label{calA_def}
\end{eqnarray} 
This represents the measure of the set of geodesics meeting $T_R$ orthogonally as viewed from the core geodesic  of the corresponding cylinder of radius $R$,  and will play an important role in our analysis.

Note that the visual area $\ca$ is the {\em same} for parallel tori, i.e. the right hand side
of (\ref{calA_def}) is independent of $R$.  For a hyperbolic cone manifold with core geodesic of length $\ell$  and cone angle $\alpha$ we have $\ca = \alpha \ell$. 
In general, $\ca$ can be expressed in terms of the complex lengths on $T$. Using the notation above we have
\begin{eqnarray}\ca  = l_b \theta_a  - l_a \theta_b
\label{visA_hol}
\end{eqnarray}
 for any positively oriented basis $a,b$ for $H_1(T;\Z)$.
 
Next, we define the {\em normalized length} of a homology class $c \in H_1(T;\R)$
on $T=T_R$ to be
\begin{eqnarray} {\hat L}(c) = { L(c) \over \sqrt{\area(T_R)}}.
\label{Lhat_def}
\end{eqnarray}
This is just the Euclidean length of $c$ after the torus $T_R$ is rescaled to have area $1$. In the case where $R=\infty$, $T$ is a horospherical torus and $\hat L$
is the same for all parallel tori, i.e. independent of the choice of horospherical cusp cross section.

Finally, let $M$ be a hyperbolic 3-manifold with tubular boundary, and let $\hat M$ denote its canonical filling. 
For any sufficiently small $r>0$, we can truncate all the ends of $\hat M$ to give a
hyperbolic manifold with tubular boundary consisting of disjoint, embedded tubular tori of tube radius $r$.
Then we define the {\em tube radius} $\hat R$ of $\hat M$ to be the supremum of all such $r$. 

Note that if $M$ has boundary components with tube radii $R_1, \ldots, R_k$
then $\hat R$ is larger than $R = \min(R_1, \ldots, R_k)$. Further, 
for any $r<\hat R$ we can truncate $\hat M$ to obtain a manifold $M$ with tubular
boundary such that all boundary components have tube radius $r$.

\bigskip

\subsection{Infinitesimal Rigidity keeping Dehn surgery coefficients fixed}
\hfill\\
Now assume we have a hyperbolic 3-manifold $M$ with tubular boundary such that
each boundary component has tube radius at least 
$R_0 = \arctanh(1/\sqrt{3}) \approx 0.65848.$ Then we can use our harmonic deformation
theory to prove an infinitesimal rigidity theorem for nearby hyperbolic structures:

\begin{theorem} Let $M$ be a compact, orientable hyperbolic 3-manifold with tubular boundary 
such that each boundary component has tube radius at least 
$R_0 = \arctanh(1/\sqrt{3}) \approx 0.65848.$ Then there are no 
infinitesimal deformations of this hyperbolic structure keeping the Dehn Surgery coefficients constant.
\label{inf_rig_rel_DS}
\end{theorem}
\begin{proof}
Each infinitesimal deformation of the holonomy of a boundary torus $T$
is represented, in a neighborhood of $T$, by a standard form $\omega_0$ with real part $\eta_0$. 
We can write $\omega_0$ as a linear combination  
\begin{eqnarray}
\omega_0 = {s \omega_m} + (x+iy){\omega_l} \qquad s,x,y \in \R
\label{std_form_decomp}
\end{eqnarray}
of the forms $\omega_m$ and $\omega_l$ given in (2) and (3) of \cite{HK2}.
(The vector fields $e_2, e_3$ in \cite{HK2} are chosen in the directions of the principal curvatures on the tubular boundary.) 
Then the effect of $\omega_0$ on complex length $\cl$ of any closed peripheral curve
is given in \cite[Lemma 2.1]{HK2} by:
\begin{eqnarray} {d \over dt} (\cl) = -2s \cl + 2(x+iy)\Re(\cl)
\label{var_cL}
\end{eqnarray}
and this formula extends by linearity to give the variation in the complex length of any element of $H_1(T;\R)$.

If the Dehn surgery coefficient $c \in H_1(T;\R)$ is {\em fixed}, then the complex length of $c$
is $\cl  = 2\pi i$. Hence
$${d \over dt} (\cl) = -2s \cl = 0$$ and $s=0$, so the $\omega_m$ term vanishes.
The contribution to the boundary term $b_R(\eta_0,\eta_0)$ defined in (\ref{b_defn}) from the boundary torus $T$ was computed explicitly in \cite[p382]{HK2}. Here this simplifies to
$$ -(x^2 + y^2) 
 {\sinh R \over \cosh R}
 \biggl(  2 + {1\over \cosh^2(R)}\biggr) \area(T) \le 0,$$
 where $R$ is the tube radius of $T$.
 
 If $M$ has multiple boundary components, then for any  infinitesimal deformation keeping the Dehn Surgery coefficients constant,
 $b_R(\eta_0,\eta_0)$ is a sum of terms of this form so $b_R(\eta_0,\eta_0) \le 0$. But for any non-trivial infinitesimal deformation,
$b_R(\eta_0,\eta_0) > 0$ by Corollary \ref{positivestandard}, and we conclude that the
infinitesimal deformation is trivial. 
\end{proof}

\subsection{Local parametrization by Dehn surgery coefficients}
\hfill\\
Let $M$ be a compact, orientable hyperbolic 3-manifold 
with tubular boundary, and suppose that the tube radius of each boundary component 
is {\em finite} and at least $R_0 = \arctanh(1/\sqrt{3}) \approx 0.65848.$ 
Let ${\cal R}$ denote the {\em character variety} of representations $\pi_1(M)\to PSL_2(\C)$ up to conjugacy
(see \cite{CS}, \cite{BZ}). First we describe the local structure of the algebraic variety ${\cal R}$ near the holonomy representation $\rho_0: \pi_1(M) \to PSL_2(\C)$ of $M$.
(Throughout this section we abuse notation, by using the same symbol for a representation and its image in the character variety.) 

\begin{theorem} Let $\rho_0: \pi_1(M) \to PSL_2(\C)$ be the holonomy representation for a compact, orientable hyperbolic 3-manifold $M$ with tubular boundary 
such that the tube radius of each boundary component 
is finite and at least $R_0 = \arctanh(1/\sqrt{3})$.
Then the character variety $\cal R$ is a smooth manifold near $\rho_0$, of complex dimension equal to the number of boundary components of $\bd M$.  Further, there is a smooth local parametrization of ${\cal R}$ near $\rho_0$ by the complex lengths of the Dehn surgery coefficient $c_0 \in H_1(\bd M; \R)$ corresponding to $\rho_0$.
\label{rep_space_smooth}
\end{theorem}

\begin{proof}
 For each representation $\rho$ near $\rho_0$, the complex length 
of peripheral curves extends to a well-defined $\R$-linear map 
$$\cl_\rho : H_1(\bd M; \R) \to \C^k,$$
where $k$ is the number of tori in $\bd M$.
 Define $F : {\cal R} \to \C^k$ by taking the complex lengths of the (initial) Dehn surgery coefficient  $c_0 \in H_1(\bd M;\R)$ for $M$:
 $$ F(\rho) = \cl_\rho(c_0).$$
As in \cite[section 4]{HK1}, the infinitesimal rigidity theorem (Theorem \ref{inf_rig_rel_DS}) implies that the derivative
 $d F_{\rho_0} : T_{\rho_0} {\cal R} \to \C^k$ has trivial kernel. Hence the Zariski tangent space 
 $T_{\rho_0} {\cal R}$ has complex dimension
 $\dim_\C  T_{\rho_0} {\cal R} \le k$.  However, by \cite[Theorem 5.6]{thnotes}, we know that
 ${\cal R}$ has complex dimension $ \ge k$ at ${\rho_0}$, so
 $\dim_\C  T_{\rho_0} {\cal R} \ge k$. 
 Hence $\dim_\C  T_{\rho_0}{\cal R} = k$ and
 ${\cal R}$ is a smooth manifold near ${\rho_0}$ of complex dimension $k$.
Further, the inverse function theorem then implies that $F$ is a local diffeomorphism.
\end{proof}
 
 Next, we show that the Dehn surgery coefficients give a smooth local parametrization near $\rho_0$.

\begin{theorem} Let $\rho_0: \pi_1(M) \to PSL_2(\C)$ be the holonomy representation for a compact hyperbolic 3-manifold $M$ with tubular boundary 
such that the tube radius of each boundary component 
is finite and at least $R_0 = \arctanh(1/\sqrt{3})$.
Then there is an open neighborhood $U$ of $\rho_0$ in $\cal R$ such that for each $\rho \in U$ there is a well defined Dehn surgery
coefficient $c(\rho)\in H_1(\bd M;\R)$, and the map $c: U \to H_1(\bd M; \R)$ is a diffeomorphism onto its image.
\label{DS_param}
\end{theorem}

\begin{proof} For each representation $\rho$ in a neighborhood $V$ of $\rho_0$
in ${\cal R}$, 
complex lengths define an $\R$-linear map 
$$\cl_\rho : H_1(\bd M;\R) \to \C^k$$
where $k$ is the number of tori in $\bd M$. We can regard this as a function of
two variables:
$$\cl : V \times H_1(\bd M;\R) \to \C^k, \quad {\rm~where~} \cl(\rho,c)=\cl_\rho(c).$$

For each $\rho$ near $\rho_0$, the corresponding Dehn surgery coefficient $c$ 
 is defined by the equation
\begin{eqnarray} \cl(\rho,c)=(2\pi i, \ldots, 2\pi i). \label{newDSeqn}
\end{eqnarray}
For $\rho=\rho_0$ this has a unique solution $c_0$. 
Differentiating the equation (\ref{newDSeqn}) at $(\rho_0,c_0)$
gives the linearized equation satisfied by tangent vectors $(\dot \rho, \dot c)$ to the solution space of (\ref{newDSeqn}):
$${\partial \cl \over \partial \rho}\biggr |_{(\rho_0,c_0)} \dot\rho + {\partial \cl \over \partial c}\biggr |_{(\rho_0,c_0)} \dot c = 0.$$
Since $\cl(\rho,c)$ is a linear function of $c$ this gives:
$$\cl_{\rho_0}(\dot c) = -{\partial \cl \over \partial \rho}\biggr  |_{(\rho_0,c_0)} \dot \rho.$$
Now the right hand side is $-dF_{\rho_0} (\dot \rho)$
where $F : {\cal R} \to \C^k$ is defined by $F(\rho) = \cl_\rho(c_0)$
as in the proof of Theorem \ref{rep_space_smooth}. So this can be written
$$ \cl_{\rho_0}(\dot c) = -dF_{\rho_0} (\dot \rho).$$
Since $\cl_{\rho_0}$ is invertible and $dF_{\rho_0}$ is invertible by Theorem \ref{rep_space_smooth}, 
this has a unique solution $\dot c$ for any $\dot \rho$ and the map
$\dot \rho \mapsto \dot c$ is invertible. 
Hence, by the implicit function theorem, (\ref{newDSeqn}) has a unique solution 
$c=c(\rho)$ for any $\rho$ near $\rho_0$ and the map 
$\rho \mapsto c(\rho)$ is local diffeomorphism. Thus the Dehn surgery coordinates give a smooth local parametrization of ${\cal R}$ near $\rho_0$.
\end{proof}

This completes the proof of  Theorem \ref{bdrylocalparam}.

\medskip
\noindent{\bf Remark:}  A hyperbolic structure with {\em infinite} tube radius $R_j =\infty$ along some components $T_j$ of $\bd M$
can be filled in to give a hyperbolic structure with complete cusps corresponding to  these components. At such structures, the corresponding complex lengths $\cl_j$ are zero
and the corresponding Dehn surgery coefficient is defined to be $c_j = \infty$. In this case,  a sign for $\cl_j$  cannot be chosen to vary continuously for nearby structures.  
However, the results of Theorem \ref{rep_space_smooth}
and Theorem \ref{DS_param} extend to this situation provided
the complex length is regarded as a function
$\cl :  \oplus_j   H_1(T_j;\R) \to (\C/\pm 1)^k$, and the Dehn surgery coefficient
as an element of $\oplus_j \hat H_1(T_j,\R)$ where 
$\hat H_1(T_j,\R) = (H_1(T_j,\R) \cup \infty)/\pm 1$.

\subsection{An effective version of the hyperbolic Dehn surgery theorem}
\hfill\\
Consider a complete, finite volume, orientable hyperbolic 3-manifold with cusps, diffeomorphic to  the interior of a compact 3-manifold $X$ with 
boundary consisting of $k$ tori ${\bd_1 X},\ldots, {\bd_k} X$.
Given a homology class $c=(c_1, \ldots, c_k) \in H_1(\bd X; \R) = \oplus_j H_1(\bd_j X; \R)$, we consider deformations of the hyperbolic structure with Dehn surgery coefficients
varying ``radially'':
\begin{eqnarray}{2\pi \over \alpha}c,   ~~~~~0 <\alpha \le 2\pi,
\label{ds_deform}\end{eqnarray}
where $\alpha$ is a smooth increasing function of a parameter $t$.
Then the complex length of the homology class $c_j$ for a given value of $t$ is
\begin{eqnarray} \cl(c_j) = \alpha(t) i. \label{cl_deform}\end{eqnarray}

We want to show that we can deform the hyperbolic structure and increase $\alpha$ to $2\pi$, provided the normalized lengths of the surgery coefficients $c_j$ are sufficiently large. 
By Thurston's original Dehn surgery theorem \cite{thnotes}, we can always increase $\alpha$ from $0$ (corresponding to the complete hyperbolic structure on the interior of $X$) to some small positive value.  The local parametrization in Theorem \ref{DS_param} shows that we can always increase $\alpha$ slightly, so the set of attainable $\alpha$ is an open subset of $(0,2\pi]$.   Then we need to control the change in geometry during the deformation, and guarantee that no degeneration of hyperbolic structures occurs before $\alpha=2\pi$ is reached.

Here is a brief outline of the argument in this section. Let $\hat M_t$ denote the filled hyperbolic manifold corresponding to parameter $t$.
By removing disjoint open tubes around the ends of $\hat M_t$ we obtain a smooth family of  a hyperbolic manifolds $M_t$ with tubular boundary. First we use the positivity condition in Corollary \ref{positivestandard} provided by our harmonic deformation theory to control the variation in complex lengths of curves on $\bd M_t$.  This leads to Proposition \ref{dvdt_inequal}, which gives differential inequalities on the {\em total visual area} $\ca$ of the boundary $\bd M_t$ (i.e. the sum of visual areas of all the boundary components). Then,  in Theorem \ref{vis_area_est}, we apply tube packing arguments to obtain a crucial estimate relating $\ca$ to the tube radius $\hat R$ of $\hat M_t$ (as defined at the end of section \ref{geom_tub}). This shows that good control on $\ca$ 
throughout a deformation will guarantee that the tube radius $\hat R$ stays bounded away from zero. Integrating the differential inequalities for $\ca$ shows that such control can be obtained provided that the normalized length of the surgery coefficient is sufficiently large (Theorem \ref{tube_radius_bounded}). In Theorem \ref{alnondeg}  we use geometric limit arguments to show that this control on the tube radius, together with bounds on the volume (Lemma \ref{schlafli}) and injectivity radius of the boundary (Lemma \ref{injrad}) imply that no degeneration of the hyperbolic manifolds $M_t$ can occur before $\alpha =2\pi$ is reached. Finally, 
combining Theorems \ref{tube_radius_bounded} and \ref{alnondeg}  gives the main results:  Theorems \ref{shapethm} and \ref{multipleshapethm} of the introduction.

Suppose the hyperbolic manifold $M_t$ has tubular boundary consisting of tori
$T_1, \ldots , T_k$ with tube radii $R_1, \ldots , R_k$. 
We  first choose
a harmonic representative $\omega$ for the infinitesimal deformation as in Theorem \ref{negativebR}. Then  
near each boundary torus $T_j$ we can write 
the standard part of $\omega$ in the form:
\begin{eqnarray} \omega_0 =  s_j \omega_m + (x_j+iy_j) \omega_l 
\qquad (s_j,x_j,y_j \in \R), \label{stdform}\end{eqnarray}
or 
\begin{eqnarray} \omega_0 = s_j (\omega_m + (X_j+iY_j) \omega_l )\end{eqnarray} 
where 
\begin{eqnarray} x_j = X_j s_j, \qquad y_j = Y_j s_j.\end{eqnarray} 

The coefficients $s_j,X_j,Y_j$ completely describe the variation in the
holonomy of the torus $T_j$ for any ``radial'' deformation of Dehn surgery coefficients
as in (\ref{ds_deform}).
In particular, using (\ref{cl_deform}) and (\ref{var_cL}) with $\cl = \cl(c_j)$, we see that
\begin{eqnarray} s_j = -{1 \over 2 \alpha} {d\alpha \over dt} 
\label{s_value}\end{eqnarray}
so all $s_j$ are all equal and
depend only on the logarithmic derivative of $\alpha$ with respect to $t$.

Now we choose a parametrization where $\alpha$ is an {\em increasing} function of $t$, 
with $\displaystyle s_j = s= - {1 \over 2 \alpha} {d \alpha \over dt} <0$ for all $j$.  
Using the crucial positivity property in Corollary \ref{positivestandard},
we obtain estimates on the size of the coefficients $X_j, Y_j$.
In particular, we have:

\begin{proposition} \label{X^2_prop}
Let $\ca_j$ be the visual area of $T_j$ and let $\ca = \sum_j \ca_j$ be the total visual area of  $\bd M_t$.  Then  if $R = \min(R_1, \ldots , R_k) \ge R_0=\arctanh(1/\sqrt{3})$, 
\begin{eqnarray}
\sum_j {\ca_j \over \ca} (X_j+\xi)^2 \le w^2, 
\label{X^2_bound}
\end{eqnarray}
where
\begin{eqnarray}
\xi = { 1\over  \sinh^2 \! R \,(2 \cosh^2 \! R + 1)} \text{  and  }
 w = {2 \cosh^2 \! R \over  \sinh^2 \! R \,(2 \cosh^2 \! R + 1)}.
 \label{defn_consts}
 \end{eqnarray}
\end{proposition}

\begin{proof}
By Corollary \ref{positivestandard}, $\eta_0 = \Re \omega_0$ satisfies
the positivity property 
$$0 \le b_R(\eta_0,\eta_0) = \int_{\bd M} \hat * D\eta_0 \wedge \eta_0.$$
Now this integral breaks up into a sum of integrals over the boundary tori $T_j$, so
$$0 \le b_R(\eta_0,\eta_0) = \sum_j \int_{T_j} \hat * D\eta_0 \wedge \eta_0.$$
A priori, some of these boundary integrals could be negative, and this makes the argument more complicated in the case of multiple boundary components.

Note that we have some flexibility in the choice of the tube radii $R_j$. In particular, 
by adding collars on boundary components we can decrease any $R_j$, 
so {\em we may assume that $R_j = R$ for each $j$}.
By explicit calculations as in \cite[p383]{HK2} we then obtain:
$$b_R(\eta_0,\eta_0) = \sum_j \left( a (X_j^2 + Y_j^2) + b X_j + c \right) \, \ca_j s^2$$
where $\ca_j$ 
is the visual area of $T_j$, $\displaystyle s = -{1 \over 2 \alpha} {d \alpha \over dt}$ and 
\begin{eqnarray}
a =  {-  \sinh^2 \! R \over \cosh^2 \! R} \left( 2 \cosh^2 \! R + 1 \right), ~~~~
b = {-2 \over\cosh^2 \! R}, ~~~~
c = {2 \cosh^2 \! R -1  \over  \sinh^2 \! R \cosh^2 \! R}.
\label{defn_abc}
\end{eqnarray}

By the completing the squares we obtain
$$0 \le b_R(\eta_0,\eta_0)=
\sum_j a\left( \left(X_j + {b\over 2 a}\right)^2 + Y_j^2\right) \ca_j s^2 +
\left( {4 a c - b^2 \over 4 a}\right) \ca_j s^2 .$$
Since $a<0$ this gives
\begin{eqnarray}
b_R(\eta_0,\eta_0) \le \left( {4 a c - b^2 \over 4 a}\right) s^2 \ca
= \left( { 4 \cosh R \over  \sinh^3 \! R \,(2 \cosh^2 \! R + 1)}\right) s^2 \ca,
\end{eqnarray}
and
\begin{eqnarray}
\quad \sum_j  \left(X_j + {b \over 2 a}\right)^2 \ca_j 
\le \left( {b^2 - 4 a c  \over 4 a^2}\right) \ca,
\end{eqnarray}
where we write $\ca = \sum_j \ca_j$.
Computing $\displaystyle{b \over 2a} = \xi$ and 
$\displaystyle{b^2 - 4 a c  \over 4 a^2} = w^2$
using (\ref{defn_abc}) gives the result.
\end{proof}

Combining this result with equation (\ref{var_cL}) gives us control on the holonomy of peripheral curves:
the variation in the complex length $\cl$ of any homology class on a boundary torus $T_j$ is given by
\begin{eqnarray}{d \over dt} (\cl) = -2s \cl  + 2 (x_j+iy_j) \Re(\cl) 
= {1 \over \alpha}{d\alpha \over dt} \cl  + 2 (x_j+iy_j) \Re(\cl).
 \label{dClength}
\end{eqnarray}

We now use this to estimate the variation in the visual area $\ca_j$ of $T_j$. 
Choose an oriented basis $a,b$ for $H_1(T_j;\Z)$ with complex lengths $l_a+  i\theta_a, l_b+ i\theta_b$. Then applying the formula (\ref{dClength}) to 
 $a,b$ gives
\begin{equation} 
{d  l_a \over dt }  +   i {d  \theta_a \over dt } = 
{1 \over \alpha}{d\alpha \over dt} (l_a + i  \theta_a ) +2(x_j+iy_j) l_a  \label{dot_mu}
\end{equation}
and
\begin{equation}  {d  l_b \over dt }  +   i  {d  \theta_b \over dt }= 
{1 \over \alpha}{d\alpha \over dt} (l_b +  i  \theta_b) +2(x_j+iy_j) l_b  . \label{dot_lambda}
\end{equation}
By differentiating equation (\ref{visA_hol}), it 
follows that
\begin{eqnarray} {d\over dt} (\ca_j) = 
{d\over dt} (l_b \theta_a  - l_a \theta_b)
=2\ca_j\left( {1 \over \alpha}{d\alpha \over dt} +x_j\right) 
\label{var_visual_area}
\end{eqnarray}
or 
\begin{eqnarray}
{d\over dt} (\ca_j) 
= (2-X_j) {\ca_j \over \alpha} \,{d\alpha \over dt}.
\end{eqnarray}
where the $X_j$ satisfy the inequality (\ref{X^2_bound}).

In the following argument, the {\em total visual area of the boundary} 
$\ca = \sum_j \ca_j$ will play a crucial role.  To control the behavior of $\ca$, we will consider the variation of 
\begin{eqnarray} v_j= {\ca_j \over \alpha^2} \text{ and } v = {\ca \over \alpha^2} = \sum _j v_j.
\label{defn_v}\end{eqnarray}
Note that these quantities only depend on the canonical filling $\hat M_t$; any truncation
$M_t$ of $\hat M_t$ with tubular boundary gives the same values for 
$\ca_j$, $\ca$, $v_j$ and $v$.

First we examine the limiting behavior of $v_j$ for our family of hyperbolic  manifolds with tubular boundary $M_t$, where we fix a homology class $c_j \in H_1(T_j;\R)$
and vary the Dehn surgery coefficients ${2\pi \over \alpha} c_j$ along a ray going out to 
$\infty$ as $\alpha \to 0$.
The homology class $c_j$ has complex length $\cl(c_j)= \alpha i$
where $\alpha$ is the deformation parameter. Hence, by equation 
(\ref{L_vs_cL}),  the Euclidean length of $c_j$ on $T_j$
is $L_j=L(c_j) = \alpha \sinh R_j$ where $R_j$ is the tube radius of $T_j$, so  
$${\area(T_j) \over L_j^2 } = { \ca_j \sinh R_j \cosh R_j \over \alpha^2 \sinh^2(R_j)}
= v_j \coth R_j .$$
As $\alpha \to 0$, the hyperbolic structures converge to the complete hyperbolic structure on the interior of $X$ and $R_j \to \infty$.  Hence
$$v_j = {\area(T_j) \over L_j^2} \tanh R_j \to  {1 \over \hat L_j^2},$$
and
$$v = \sum_j v_j \to  \sum_j {1 \over \hat L_j^2},$$
where $\hat L_j = \hat L(c_j)$ is the normalized Euclidean length of $c_j$, 
as defined in (\ref{Lhat_def}), on a horospherical cross section for the 
complete hyperbolic structure on the interior of $X$. 

\noindent{\bf Remark:} The quantity $\hat L = \hat L (c) $ defined by
\begin{eqnarray}
{1 \over {\hat L}^2} = \sum_j {1 \over {\hat L_j}^2}
\label{gen_hatL}
\end{eqnarray}
seems to be a useful analogue of the normalized length in the one-cusped case, and
we will also call it the {\em normalized length} of the homology class 
$c \in H_1(\bd X;\R)$
in the multi-cusped case. 
Its reciprocal $1/\hat L$ gives a good measure of the distance from the complete hyperbolic structure on the interior of $X$ to the hyperbolic structure $M(c)$ with Dehn surgery coefficient $c$.

\smallskip
Differentiating $v_j$ 
using equation (\ref{var_visual_area})  gives
\begin{eqnarray}
{dv_j \over dt} &=& {d \over dt}\left({\ca_j \over \alpha^2}\right) 
= {1\over \alpha^2}\left( {d\ca_j \over dt} - {2 \ca_j \over \alpha}{d\alpha \over dt}\right) 
= {2 \ca_j x_j\over \alpha^2} = 2 v_j x_j
\end{eqnarray}
and 
\begin{eqnarray}
{1 \over v} {dv \over dt} = 2 \sum_j {v_j \over v} x_j = 
-{1 \over \alpha} {d\alpha \over dt} \sum_j {v_j \over v} X_j.
\label{dvdt}
\end{eqnarray}
Combining this with Proposition \ref{X^2_prop} gives our basic
differential inequalities for $v$.

\begin{proposition} \label{dvdt_inequal}
Let $v=\ca/\alpha^2 = v_1 + \ldots + v_k$ where $v_j = \ca_j/\alpha^2$. 
Then $v$ satisfies the differential inequalities
\begin{eqnarray}
{1 \over \sinh^2 \hat R} {1 \over \alpha} {d\alpha \over dt} 
\ge  {1 \over v} {dv \over dt}  \ge
-{1 \over \sinh^2 \!  \hat R}\left({2 \cosh^2 \!  \hat R -1 \over 2 \cosh^2 \!  \hat R +1}\right)
 {1 \over \alpha} {d\alpha \over dt},
\label{multicusp_inequal}
\end{eqnarray}
provided the tube radius $\hat R$ of the canonical filling $\hat M_t$
is larger than $R_0$.
 Further, $v$ satisfies the initial condition
\begin{eqnarray}
\lim_{\alpha \to 0} v = \sum_j {1 \over \hat L_j^2},
\label{v_initial}
\end{eqnarray}
where $\hat L_j = \hat L(c_j)$ is the normalized Euclidean length of $c_j$, 
on a horospherical cross section for the 
complete hyperbolic structure on the interior of $X$. 
\end{proposition}

\begin{proof} Since $\hat R > R_0$, we can truncate $\hat M_t$ to give a
hyperbolic manifold with tubular boundary $M_t$ such that all components
of $\bd M_t$ have tube radius at least $R$, where $\hat R > R \ge R_0$.

Using the Cauchy-Schwartz inequality and (\ref{X^2_bound}) then gives
\begin{eqnarray*} \left(   \sum_j {v_j \over v}  (X_j+{\xi})   \right)^2
&=& \left(\sum_j \left({v_j \over v}\right)^{1/2} \cdot 
\left({v_j \over v}\right)^{1/2} (X_j+{\xi}) \right)^2  \\
&\le& \sum_j {v_j \over v} \sum_j  {v_j \over v} (X_j+{\xi})^2  
=  \sum_j  {v_j \over v} (X_j+{\xi})^2 \le w^2,\end{eqnarray*}
since $\sum_j v_j = v$.
Hence, 
\begin{eqnarray} -w - {\xi}
\le  \sum_j {v_j \over v} X_j 
\le w - {\xi}. \label{X_bound}
\end{eqnarray}
Writing out $w$ and ${\xi}$ in terms of $R$ gives
$$
-{1 \over \sinh^2R}
\le  \sum_j {v_j \over v} X_j  \le
{1 \over \sinh^2 \! R}\left({2 \cosh^2 \! R -1 \over 2 \cosh^2 \! R +1}\right).
$$
Multiplying through by the negative number $\displaystyle-{1 \over \alpha} {d\alpha \over dt}$
and recalling, from (\ref{dvdt}), that
$${1 \over v} {dv \over dt} =  -{1 \over \alpha} {d\alpha \over dt} \sum_j {v_j \over v} X_j$$
gives the inequality
\begin{eqnarray*}
{1 \over \sinh^2R} {1 \over \alpha} {d\alpha \over dt} 
\ge  {1 \over v} {dv \over dt}  \ge
-{1 \over \sinh^2 \! R}\left({2 \cosh^2 \! R -1 \over 2 \cosh^2 \! R +1}\right)
 {1 \over \alpha} {d\alpha \over dt}.
\end{eqnarray*}
Since these inequalities hold for all $R$ such that $R_0\le R<\hat R$, they also 
hold when $R$ is replaced by $\hat R$. This gives (\ref{multicusp_inequal}).
The initial conditions for $v$ was already derived above.
\end{proof}

Next, let $M$ be a hyperbolic 3-manifold with tubular boundary, and let $\hat M$ denote 
its canonical filling. Then the tube packing arguments in the proof of  \cite[Theorem 4.4]{HK2}  give us the following crucial estimate relating the total visual area of $\bd M$ to the tube radius of $\hat M$.

\begin{theorem} \label{vis_area_est}
Let $M$ be a compact, orientable hyperbolic 3-manifold with tubular boundary,
and let $\ca$ be the total visual area of $\bd M$.
Then
\begin{eqnarray} \ca \ge h(\hat R),
\label{calAbound}
\end{eqnarray}
where $\hat R$ is the tube radius of the filled manifold $\hat M$, and $h$ is the function given by
\begin{eqnarray}
h(r) = 3.3957 {\tanh r \over \cosh(2r)}.
\label{h_def}
\end{eqnarray}
\end{theorem}

\begin{proof} We briefly recall the argument from \cite{HK2}.
If we expand tubes around the ends of $\hat M$ at the same rate, then these first bump
when the tube radius is $\hat R$. Suppose that the tube bounded by a torus $T_i$ bumps
into the tube bounded by a torus $T_j$ when the tube radius reaches $\hat R$; possibly with $i=j$.
Then the tube packing arguments 
from \cite{HK2} 
show that  $T_i \cup T_j$ contains two
open ellipses meeting only at the bumping point, each with semi-major axes
$$a = {0.980258 \sinh \hat R \cosh \hat R \over  \cosh(2 \hat R)}
\text{ and }
b={ \sinh \hat R \sinh \hat R \over \sinh(2 \hat R)}$$
and hence
of area  
$$A_e= \pi a b = \displaystyle{0.980258 \, \pi \sinh^2 \! \hat R\over 2  \cosh(2\hat R)}.$$
(These two ellipses are in the same torus $T_i$ if $j=i$; otherwise
there is one ellipse in $T_i$ and one ellipse in $T_j$.)
Further, 
the packing density for these ellipses is at most ${\pi \over 2 \sqrt{3}}$,
so it follows that
$$\area(T_i \cup T_j) \ge {4 \sqrt{3} \over \pi} A_e \ge 3.3957 { \sinh^2 \! \hat R \over \cosh(2\hat R)}$$
and the total visual area of $\bd M$ satisfies
$$\ca \ge {\area(T_i \cup T_j) \over \sinh \hat R \cosh \hat R}
\ge  3.3957 { \sinh \hat R \over \cosh \hat R  \cosh(2\hat R)} = h(\hat R).$$
\end{proof}

Now we can apply the same arguments as in \cite[section 5]{HK2}, but with
the tube radius condition $\hat R > 0.531$ replaced by $\hat R > 
R_0 = \arctanh(1/\sqrt{3}) \approx 0.65848.$ Since $h(r)$ is a decreasing function
for $r \ge R_0$ it follows from (\ref{calAbound}) that if initially the tube radius 
satisfies $\hat R > R_0$ and we know that $\ca < h(R_0)$ throughout a deformation, then
$\hat R > R_0$ throughout the deformation.

Next we use the control on $\ca$ given by Proposition \ref{dvdt_inequal} 
and the inequality (\ref{calAbound}) to show that the tube radius $\hat R$
of $\hat M_t$ stays bounded below throughout any deformation as in (\ref{ds_deform})
with $0 \le \alpha(t) \le \alpha_0 \le 2\pi$, provided that the normalized
lengths of the Dehn surgery coefficients $c_j \in H_1(T_j;\R)$ are sufficiently large.

First note that the inequalities (\ref{multicusp_inequal}) are {\em exactly equivalent} to the differential inequalities for $u = 1/v$ obtained in \cite{HK2}
in the one cusped case:
\begin{eqnarray*}
-{1 \over \sinh^2 \hat R} {1 \over \alpha} {d\alpha \over dt} 
\le  {1 \over u} {du \over dt}  \le
{1 \over \sinh^2 \!  \hat R}\left({2 \cosh^2 \!  \hat R -1 \over 2 \cosh^2 \!  \hat R +1}\right)
 {1 \over \alpha} {d\alpha \over dt}.
\end{eqnarray*}

Now we choose a parametrization with  $t=\alpha^2$.  Then this
becomes
\begin{eqnarray}
-{1 \over \sinh^2 \hat R} {u \over 2 \alpha^2} 
\le  {du \over dt}  \le
{1 \over \sinh^2 \!  \hat R}\left({2 \cosh^2 \!  \hat R -1 \over 2 \cosh^2 \!  \hat R +1}\right)
 {u \over 2 \alpha^2},
 \label{dudt_inequal}
\end{eqnarray}

We analyze this as in  \cite{HK2} by introducing a new variable $z=\tanh(\rho)$
where $h(\rho)= \ca$ and $\rho \ge R_0$. 
Note that $\rho$ is defined
whenever $\ca \le h(R_0)$ and if $\hat R \ge R_0$ then (\ref{calAbound}) implies
that $R_0 \le \rho \le \hat R$. This allows us to replace $\hat R$ by $\rho$ in the 
inequality (\ref{dudt_inequal}).

Now we define functions
\begin{eqnarray}H(z) = {1 \over \ca} = {1 \over h(\r)} =
{1+z^2 \over {3.3957 z (1-z^2)}},
\label{Hdef}\end{eqnarray}
\begin{eqnarray}G(z) = {H(z)\over 2}{1-z^2 \over z^2} = {1+z^2 \over 6.7914
~ z^3},
\label{Gdef} \end{eqnarray}
and
\begin{eqnarray}\tilde G(z) = {H(z)\over 2}{(1-z^2)(1+z^2) \over z^2(3-z^2)} =
{(1+z^2)^2 \over 6.7914 ~ z^3\,(3-z^2)}.\label{GGdef}\end{eqnarray}

Then the differential inequality (\ref{dudt_inequal}) for $u$, with $\hat R$ replaced
by $\rho$, becomes
$$-G(z) \le {d u \over dt} \le \tilde G(z)$$
and putting $u=tH(z)$ gives the differential inequalities for $\displaystyle{dz \over dt}$
obtained in  \cite[equation (51)]{HK2}:
\begin{equation}
{H'(z) \over H(z)-\tilde G(z)} {dz \over dt} \le -{1\over t} \le
{H'(z) \over H(z)+G(z)} {dz \over dt}.
\label{dz_inequal}
\end{equation}
To solve these differential inequalities we write:
$${H'(z) \over H(z)-\tilde G(z)}= \tilde F(z) + {1 \over 1-z},~~~~
{H'(z) \over H(z)+G(z)} = F(z) + {1 \over 1-z}.$$
Then $$F(z) = -{(1+4z+6z^2+z^4)\over(z+1)(1+z^2)^2}$$
is integrable on the interval  $0 \le z \le 1$ and
$${\tilde F}(z) = -\frac{z^6+7 z^4+12 z^3-9 z^2-4 z+1}{(z+1)
   \left(z^2+1\right) \left(z^2-2 z-1\right) \left(z^2+2 z-1\right)}$$
 is integrable on the interval $\sqrt{2}-1+\varepsilon \le z \le 1$ for each $\varepsilon>0$.

Now we integrate (\ref{dz_inequal}) with respect to $t$ from $t_1$ to $\tau$ where $0<t_1< \tau<1$,
and $z(\tau) \ge \tanh(R_0) = 1/\sqrt{3}$.
Carefully taking a limit as $t_1\to 0, z_1=z(t_1) \to 1$ as in \cite{HK2} gives estimates
on the time $\tau=\alpha^2$ taken to reach a given value of
$\ca = \ca(z) = {1 \over H(z)}$ where $z=z(\tau)$.
Define $\hat L$ by
$${1 \over {\hat L}^2} = \sum_j {1 \over {\hat L_j}^2}$$
where $\hat L_j = \hat L(c_j)$ is the normalized length of the homology class $c_j \in H_1(T_j;\R)$
for the complete hyperbolic structure on the interior of $X$, as in (\ref{gen_hatL}). Then we obtain
\begin{eqnarray}
\tilde f(z) \ge {\alpha^2 \over {\hat L}^2} \ge f(z),
\label{Lhat_bound}
\end{eqnarray}
where
\begin{eqnarray}
\tilde f(z) = 3.3957(1-z) \exp(-\int_1^z \tilde F(w)\, dw)
\label{tilde_f_def}
\end{eqnarray}
and
\begin{eqnarray}
f(z) = 3.3957(1-z) \exp(-\int_1^z F(w)\, dw).
\label{f_def}
\end{eqnarray}
These bounds are illustrated in the graph in Figure 1. 

\begin{figure}[h]
\begin{center}
\includegraphics[scale=0.9]{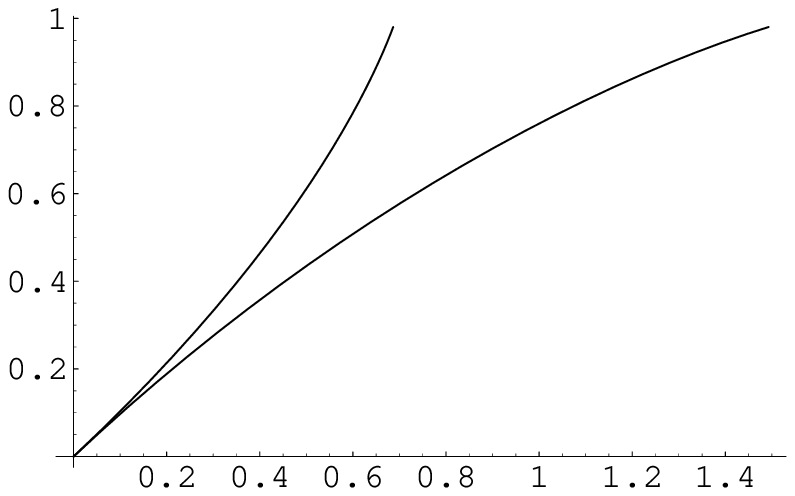}
\end{center}
\caption{Graph of $\ca$ versus $x = {\alpha^2 \over {\hat L}^2}$}\label{fig1}
\end{figure}

We conclude that  we can increase the parameter $\alpha$ from $0$ to $2\pi$,
maintaining $z=\tanh \r > z_0 = \tanh(R_0)$,
hence keeping the tube radius $\hat R \ge \r > R_0 = \arctanh(1/\sqrt{3})$ and
$\ca < h(R_0)$, provided
$$ \hat \L^2 ~> ~ {(2\pi)^2 \over 3.3957 (1-z_0)} ~
\exp\left(\int_1^{z_0} F(w) dw \right)
\approx 57.5041$$
or
$$\hat \L ~>~ \sqrt{57.5041} \approx 7.58315.$$
Thus, we have shown that as long as 
$\hat L$ satisfies this inequality then
there is a lower bound to the tube radius: 

\begin{theorem} Let $M_t$ be a smooth family of hyperbolic structures 
with tubular boundary on $X$ for $0 \le t <1$, 
with Dehn surgery coefficients ${2\pi \over \alpha(t)} c$ where 
$c= (c_1, \ldots, c_k) \in H_1(\bd X;\R)$ and $0\le \alpha(t) \le \alpha_0 \le 2\pi$.
If the normalized lengths of the surgery coefficients $\hat L_i = \hat L(c_i)$ 
satisfy
$$\sum_i {1\over \hat L_i^2} < {1 \over C^2} \text{ where } C= 7.5832,$$
then tube radius of the filled manifold $\hat M_t$ satisfies $\hat R > R_0$ for all $t$.
\label{tube_radius_bounded}
\end{theorem}

The next main result is the following analogue of  \cite[Theorem 5.4]{HK2}:

\begin{theorem}\label{alnondeg} 
Let $M_t$ be a smooth family of hyperbolic structures 
with tubular boundary on $X$ for $0\le t <1$,
with Dehn surgery coefficients ${2\pi \over \alpha(t)} c$ where 
$c= (c_1, \ldots, c_k) \in H_1(\bd X;\R)$ and 
$\alpha(t)$ is an increasing function of $t$
with $0\le \alpha(t) \le \alpha_0 \le 2\pi$. 
Suppose the  tube radius of the canonical filling $\hat M_t$
satisfies $\hat R > R_0 =\arctanh(1/\sqrt{3}) \approx 0.6585$ for $t=0$
and the total visual area of $\bd M_t$ satisfies $\ca \le h_0 = h(R_0)$ for all $t$.  
Then the manifolds $M_t$ converge geometrically as $t \to 1$ to a hyperbolic manifold $M_1$ with tubular boundary.  Further, the Dehn surgery coefficient for $M_t$ converges to the Dehn surgery coefficient for $M_1$. 
\end{theorem}

\begin{proof} We begin with some estimates on the geometry of the manifolds $M_t$.
First we study the change in volume of the {\em filled in} manifolds $\hat M_t$ with Dehn surgery type singularities. In fact, we have Schl\"afli type formula for the variation in volume: 

\begin{lemma} \label{schlafli}
Let $M_t$ be a smooth family of hyperbolic 3-manifolds with tubular boundary 
with Dehn surgery coefficients varying radially as in (\ref{ds_deform}).
Then the variation in volume $V$ of the filled manifolds $\hat M_t$
is given by 
\begin{eqnarray}
dV = - { \ca \over 2 \alpha} d \alpha,
\label{dVol}
\end{eqnarray}
where $\ca = \sum_j \ca_j $ is  the total visual area of $\bd M_t$.
In particular, the volume decreases as $\alpha$ increases. 
\end{lemma}
\noindent Remark:  In the cone manifold case, (\ref{dVol}) is just the usual Schl\"afli formula:\\
  $dV = -{1\over 2} \sum_j \ell_j d\alpha$, where $\ell_j$ is the length 
and $\alpha$ is the cone angle of the
core geodesic produced when $T_j$ is filled. 

\begin{proof}
From \cite[equation (46)]{neumann-zagier}  
(or \cite[chapter 5]{Ho}, \cite[section 4.5]{CCGLS}) 
the variation in volume is a sum of contributions $dV_j$ from the boundary tori $T_j$,
and we have
\begin{eqnarray}
dV_j =  -{1 \over 2} ( l_{b} d \theta_{a} - l_{a} d \theta_{b}),
\label{dV_schalfli}
\end{eqnarray}
if $a,b$ is any oriented basis for $H_1(T_j;\Z)$.
Now, from equations (\ref{dot_mu}) and (\ref{dot_lambda}), we have
\begin{eqnarray*}
l_{b} {d \theta_{a} \over dt} - l_{a} {d \theta_{b} \over dt}
&=& l_{b}({1\over\alpha}{d\alpha \over dt} \theta_a + 2y l_{a})-  
l_{a}({1\over\alpha}{d\alpha \over dt} \theta_b + 2 y l_{b}) \\
&=& {1\over\alpha}{d\alpha \over dt} (l_{b}\theta_{a}-l_{a}\theta_{b})= {\ca_j \over \alpha} {d\alpha \over dt}.
\end{eqnarray*}
Hence
\begin{eqnarray*}
dV = -\sum_j { \ca_j \over 2 \alpha} d \alpha = - { \ca \over 2 \alpha} d \alpha.
\end{eqnarray*}
\end{proof}

Tube packing arguments as in \cite{HK2} give the following estimate on injectivity radius of the boundary of a hyperbolic 3-manifold with tubular boundary. 

\begin{lemma}\label{injrad}
Let $M$ be a compact, orientable hyperbolic 3-manifold with tubular boundary 
such that each boundary component has tube radius at least $R>0$. Then 
there exists a constant $c(R)>0$ such that $M$ can be
truncated along tubular tori to give a hyperbolic 3-manifold with tubular boundary 
such that the injectivity radius of the Euclidean metric on each boundary component 
is at least $c(R)$.
\end{lemma}

\begin{proof} We first expand the boundary tori of $\bd M$, moving each torus inwards at the
same rate 
until it bumps into itself or another boundary component.
Let $T_1, \ldots , T_k$ be the (immersed) tubular tori obtained when this bumping occurs; these tori meet tangentially in a finite collection of points. Let $R_i$ be the tube radius of $T_i$;
then $R_i \ge R$ for all $i$.  
  
We let $N=\hat M$ denote the canonical filling of $M$, $\tilde N$ the universal covering of $N$, and $\hat N$ the metric completion of $\tilde N$. 
For each subset $Y \subset \tilde N \subset \hat N$, let $\hat Y$ denote the closure of 
$Y$ in $\hat N$.

We can regard $T_1, \ldots, T_k$ as a subsets of the canonical filling $N$. Then
each $T_i$ bounds an open tube $V_i$ in $N$ and the tubes $V_1, \ldots, V_k$ are disjoint.
 
Let $U_i$ be one (fixed) lift of $V_i$ to the universal cover
$\tilde N$ of $N$ and let $\hat U_i$ be its closure in $\hat N$.
Now consider the lifts $U$ of tubes $V_j$ to $\tilde N$ such that
$\bd U$ meets $\bd U_i$ tangentially at a point (possibly with $j=i$), and let $\hat U$ be the closure of $U$ in $\hat N$.  For each such $U$, we construct a point $q \in \hat U$
as follows: Let $p$ denote the intersection point of $\bd U_i$ and $\bd U$.
Then $q$ is the point inside $\hat U$ at distance $R$ from $p$ along the geodesic through $p$ orthogonal to $\bd U$.

\begin{claim} Let $Q$ be the collection of all points $q$ constructed as above. Then 
the distance in $\hat N$ between any two distinct points of $Q$ is at least $2R$.
\end{claim}

\begin{proof} Let $U', U'' \subset \tilde N$ be lifts of the tubes $V_1,\ldots,V_k$ whose closures 
$\hat U',\hat U'' \subset \hat N$ contain two distinct points $q' , q'' \in Q$, and
let $B(q',R)$ and $B(q'',R)$ denote the open balls in $\hat N$ of radius $R$ around $q'$  and $q''$ respectively.
Since $R \le R_j$, it follows from the triangle inequality that 
$B(q',R) \subset \hat U'$ and $B(q'',R) \subset \hat U''$. But $\hat U'$ and $\hat U''$ have disjoint interiors, hence $B(q',R)$ and $B(q'',R)$ are disjoint. This proves the claim.
\end{proof}

Now the distance from the boundary of $V_i$ to any singular point in the completion 
of $N$ is at least $R$. So we can expand $\hat U_i$ to an open tube $\hat W_i$ of radius $d_i = R_i + R$ which embeds isometrically in $\hat N$.  

The geometry of $\hat W_i$ can be described as follows.
Let $g$ be a geodesic in $\H^3$. The universal cover of
$\H^3 - g$ can be completed by adding a geodesic, $\hat g$, which projects
to $g$ in $\H^3$. (This can be thought of
as the infinite cyclic branched cover of $\H^3$ branched
over the geodesic $g$.) Let $\hat \H^3$ denote this completion. Then
there is an isometry $\phi :   \hat W_i \to  W $ to the open tube $W$ of radius $d_i$ about $\hat g$ in $\hat \H^3$.

We use this to identify $\hat W_i$ with $W \subset \hat \H^3$ and to identify the set of points $Q$ with a subset $\hat Q \subset \bd W \subset \hat \H^3$.  Since the closure $\overline W$ of $W$ is a convex subset of $\hat \H^3$, it follows from the claim above that
$d_{\hat \H^3} (q',q'') = d_{\overline W} (q',q'')  
\ge 2 R$ for all $q' \ne q'' \in \hat Q$.

Next we use the arguments of \cite{HK2} to estimate the distance measured on $\bd U_i$ between the tangency points $p$ described above. For $q \in \hat Q$, let $B_q$ denote the closed ball in $\hat \H^3$ of radius $R$ around the point $q$ 
and let $P_q$ denote the orthogonal projection of $B_q$ onto the surface at radius $R_i$ from the core geodesic of $\hat \H^3$. Since the balls $B_q$ have disjoint interiors, are of equal radius and are allcentered at the same distance $d_i = R_i + R$ from the core geodesic in $\hat \H^3$, it follows that their projections $P_q$ also have disjoint interiors.

As in \cite{HK2} we introduce cylindrical coordinates on $\H^3$ around the geodesic $g$ and lift these to cylindrical coordinates
$(r,\theta,\zeta)$ on $\hat \H^3$; here the angle $\theta$ is a well defined real number.
Then from  \cite[Lemma 4.3]{HK2}, the projection of $B_q$ 
to the $(\theta,\zeta)$ plane can be parametrized by 
$$\sinh^2\zeta \cosh^2  (R_i+R) + \sin^2 \theta \sinh^2 (R_i+R) \le \sinh^2 \! R,$$
where $R \le R_i$.

Now, as in \cite{HK2}, we have 
$$| \sinh \zeta | \le {\sinh R \over  \cosh(R_i+R) }
\le {\sinh R \over \cosh(2 R)} \le {1 \over 2\sqrt{2}},$$
hence, using convexity of the $\sinh$ function,
$$|{ \sinh \zeta}| \le S |  \zeta| \text{ where }
S= {{1 \over 2\sqrt{2}} \over \arcsinh{1 \over 2\sqrt{2}}}
< {1 \over 0.980258}.$$
Further, $$|\sin \theta| \le |\theta|,$$ so
$P_q$ contains the region:
$$(S\zeta)^2 \cosh^2(R_i+R) + \theta^2 \sinh^2(R_i+R)  \le \sinh^2 \! R$$
or 
$$\left({S \cosh(R_i+R)\over \sinh R \cosh R_i}\right)^2 (\zeta\cosh R_i)^2+
\left({ \sinh(R_i+R) \over \sinh R \sinh R_i}\right)^2 ( \theta \sinh R_i)^2 \le 1.$$
Since $\z \cosh R_i$ and $\theta\sinh R_i$ are Euclidean coordinates on the surface
in $\H^3$ at radius $r=R_i$, this
equation describes an ellipse with semi-major axes
$$a = { \sinh R \cosh R_i \over S \cosh(R_i+R)}
\text{ and }
b={ \sinh R \sinh R_i \over \sinh(R_i+R)}$$
whose interior is disjoint from the ellipses corresponding to other points $q \in \hat Q$.
Thus these project to ellipses in $T_i$ with disjoint interiors.

Now
$${1\over a} = S(\coth R + \tanh R_i) < S(\coth R+1)$$
and
$${1\over b} = \coth R + \coth R_i \le 2 \coth R < \coth R + 1 \le S(\coth R+1),$$
where $S< {1 \over 0.980258}.$
Hence the Euclidean injectivity radius of $T_i$ is larger than 
$$c(R) = \displaystyle{0.980258 \over \coth R + 1}.$$
By adding sufficiently small collars on the tori $T_i$ we obtain a hyperbolic manifold
with tubular boundary as desired.
\end{proof}

The proof of theorem \ref{alnondeg} now follows from the arguments in the proofs of 
\cite[Theorem 1.2]{HK2}  and  \cite[Theorem 3.12]{HK2} together with the volume and injectivity radius estimates given in Lemmas  \ref{schlafli} and \ref{injrad} above.
We outline the argument.

Let $M$ be a hyperbolic 3-manifold with tubular boundary whose canonical filling
$\hat M$ has tube radius at least $R_0$. Then by using Lemma \ref{injrad}
we can add standard collars on the boundary components
to obtain a hyperbolic manifold $N \subset \hat M$ with tubular boundary such that 
the boundary components of $\bd N$
are at least distance $R_0/2$ apart and have tube radius at least $R_0/2$, and whose injectivity radius at all boundary points (as defined in \cite[p388]{HK2})
is at least $c'$, where $c'$ is a positive constant depending only on $R_0$.

Thus we can truncate the manifolds $\hat M_t$ to obtain a smooth family of hyperbolic manifolds $N_t$ with tubular boundary, with uniform lower bounds on the 
tube radius and injectivity radius (as defined in \cite[p388]{HK2}) on their boundaries. 
Further, the volume of $N_t$ is at most the volume of the canonical filling
$\hat N_t = \hat M_t$, and Lemma \ref{schlafli} shows that this is decreasing throughout the deformation since $\alpha(t)$ is increasing. Hence the volumes of the $N_t$ are bounded above.

 It follows 
as in the proof of  \cite[Theorem 3.12]{HK2} that $N_t$ converge in the bilipschitz topology to a hyperbolic manifold $N_1$ with tubular boundary, and we can choose holonomy representations $\rho_t$ for $N_t$ converging to the holonomy representation $\rho_1$ for $N_1$.  Further, the convergence of the geometry on the boundary of $N_t$ implies, using (\ref{C_len_def}), that the corresponding complex length functions
${\cal L}_t$ for $\bd N_t$ converge to the complex length function ${\cal L}_t$ for $\bd N_1$.

In particular, since $\lim_{t\to 1} {\cal L}_t(c_j) = \lim_{t\to 1} \alpha(t) i \ne 0$ the limiting complex length function is not identically zero. Hence the limiting tube radius for each boundary component of $N_1$ is {\em finite}.  It follows that ${\cal L}_1$ is invertible and the Dehn surgery coefficient for each boundary component of $N_1$ is the limit of the corresponding Dehn surgery coefficients for $N_t$. This completes the proof of Theorem \ref{alnondeg}.
\end{proof}

Combining the control on tube radii given by Theorem \ref{tube_radius_bounded} with Theorem
\ref{alnondeg} shows that we can deform the complete hyperbolic structure on 
the interior of $X$
to a hyperbolic structure with Dehn surgery coefficient $c \in H_1(\bd X;\R)$ provided
the normalized length $\hat L$ of $c$ is at least  $7.5832$. This proves
our main result.

\begin{theorem} 
Consider a complete, finite volume hyperbolic structure on the interior of a compact, orientable $3$-manifold $X$ with $k \ge 1$ torus boundary components.  
Let $T_1, \ldots, T_k$ be horospherical tori which are embedded as a 
cross-sections to the cusps of the complete structure.   
Then there exists a universal constant $C = 7.5832$ such that 
there is a ``radial'' deformation from the complete hyperbolic structure on the interior of $X$
 to a hyperbolic structure with Dehn surgery coefficient 
 $c = (c_1,c_2, \cdots, c_k) \in H_1(\bd X;\R)$
 through hyperbolic structures with Dehn surgery coefficients ${2\pi \over \alpha} c$ provided the normalized lengths
$\hat L_j = \hat L(c_j)$ on $T_j$ satisfy 
$$\sum_j {1 \over \hat L_j^2} < {1 \over C^2}.$$
\label{multicusp_thm}
\end{theorem}
In particular, this implies Theorems  \ref{shapethm} and \ref{multipleshapethm} as stated in the introduction.

\smallskip
\noindent{\bf Remark:} The same result holds if some surgery coefficients $c_j$
are infinite, i.e. some cusps remain complete. The proof is essentially the same as
given above.

\subsection{Volume estimates}
\hfill\\
Finally, we use our estimates on the total visual area $\ca$ to control the change in geometry during hyperbolic Dehn filling. 
We consider a deformation as in the previous section,
and use the notation and results from that section.

The inequalities in (\ref{Lhat_bound}) give us upper and lower bounds on $\ca$ as a function of
$x = {\alpha^2 \over \hat L^2}$ provided $x \le ({2\pi \over 7.5832})^2$.
If $\hat L \ge 7.5832$ then these estimates apply throughout the deformation,
as $\alpha$ increases from $0$ to $2 \pi$. Then
by integrating the estimates, we obtain upper and lower bounds on the change in the volume $V$
of the filled in manifolds $\hat M_t$ with Dehn surgery type singularities.

Using the parametrization $t=\alpha^2$, Schl\"afli's formula (\ref{dVol}) gives
$${d \over dt} (V) 
= - {\ca \over 2\alpha} {d \alpha \over dt} 
= -{\ca \over 4 \alpha^2} = -{\ca \over 4t}.$$
Fixing $\hat L$ and putting $x = {\alpha^2 \over {\hat L}^2}={t  \over {\hat L}^2}$ we find that
the {\em decrease} in volume is
$$\Delta V  
= \int_0^{(2\pi)^2} {\ca \over 4t} \, dt
= \int_0^{\hat x} {\ca(x) \over4 x} \, dx$$
where $\hat x= {(2\pi)^2 \over {\hat L}^2}$ and $\ca(x)$ lies in the
region defined by the inequalities in (\ref{Lhat_bound}) and illustrated in Figure \ref{fig1}.

The lower bound for $\ca(x)$ is given by $\ca = \ca (z)$ where $x=\tilde f(z)$,
$dx = \tilde f'(z) dz$. Hence
$${1\over 4} \int_1^{\tilde z} {\ca(z) \tilde f'(z) \over \tilde f(z)} \,dz \le \Delta V$$
where $\tilde f ( \tilde z) = \hat x= {(2\pi)^2 \over {\hat L}^2}$. Rewriting this
using the definitions of $\tilde f$ and  $\ca= 1/H$ gives
$${1\over 4} \int_{\tilde z}^1 {H'(z) \over H(z)(H(z) - \tilde G(z))} \, dz \le \Delta V.$$

The upper bound for $\ca(x)$ is given by $\ca = \ca (z)$ where $x= f(z)$,
$dx =  f'(z) dz$. Hence
$$\Delta V \le {1\over 4} \int_1^{\hat z} {\ca(z) f'(z) \over  f(z)} \,dz$$
where $f ( \hat z) = \hat x= {(2\pi)^2 \over {\hat L}^2}$. Hence
$$\Delta V \le  {1\over 4} \int_{\hat z}^1 {H'(z) \over H(z)(H(z)+ G(z))} \, dz .$$

In particular, taking $\hat z = 1/\sqrt{3}$ gives an upper bound
$\Delta V \le 0.197816$ when ${(2\pi)^2 \over {\hat L}^2} \le f(1/\sqrt{3})$,
i.e. ${\hat L} \ge 7.5832$.

This gives us the following estimates on the changes in geometry during generalized hyperbolic Dehn filling. 

\begin{theorem} \label{geomthm}
Let $X$ be a compact, orientable $3$-manifold as in Theorem \ref{multicusp_thm}, 
and let $V_\infty$ denote the volume of the complete hyperbolic structure
on the interior of $X$.
Let $c =(c_1,\ldots,c_k)\in H_1(\bd X;\R)$
be a surgery coefficient with normalized lengths $\hat L_i = \hat L (c_i)$
satisfying
$${1 \over {\hat L}^2} = \sum_j {1 \over \hat L_j^2} < {1 \over C^2} \text{ where } C = 7.5832,$$
and let $M(c)$ be the filled hyperbolic manifold with Dehn surgery coefficient $c$.
 Then:
\begin{enumerate}
\item The decrease in volume in hyperbolic Dehn filling 
$\Delta V = V_\infty - \vol(M(c))$ satisfies
$${1\over 4} \int_{\tilde z}^1 {H'(z) \over H(z)(H(z) - \tilde G(z))} \, dz 
\le \Delta V  \le  {1\over 4} \int_{\hat z}^1 {H'(z) \over H(z)(H(z)+ G(z))} \, dz, $$
where $\hat z$ and $\tilde z$ are defined 
by $$f ( \hat z) = \tilde f ( \tilde z) =  {(2\pi)^2 \over {\hat L}^2}.$$
\item The total visual area $\ca$ in $M(c)$ satisfies
$$ \ca(\tilde z)= {1 \over H(\tilde z)} \le \ca \le \ca(\hat z) = {1 \over H(\hat z)},$$
where $\hat z, \tilde z$  are as above.
\end{enumerate}
\end{theorem}

\noindent{\bf Remark:} The estimate on the change in volume during Dehn filling in 
Theorem \ref{geomthm} is a significant improvement on  the estimate obtained in \cite{HK2}. The previous analysis
bounded the change in volume until $\ca = \alpha \ell$ reached its maximum allowed value $h_{max}$, but in this process the parameter $\alpha$ could increase beyond $2\pi$. Here we estimate the change in volume until $\alpha$ reaches $2\pi$;
 this gives the more refined estimate.

The graphs in Figures 2 and 3 illustrate the results in Theorem \ref{geomthm}. 
The dotted lines in these figures correspond to the asymptotic formulas of 
Neumann-Zagier \cite{neumann-zagier}:  as ${\hat L} \to \infty$,  
the decrease in volume is
$\Delta V \sim  {\pi^2  \over \hat L^2}$
and the visual area is
$\ca \sim {(2\pi)^2 \over \hat L^2}.$
\begin{figure}[h]
\begin{center}
\includegraphics[scale = 0.85]{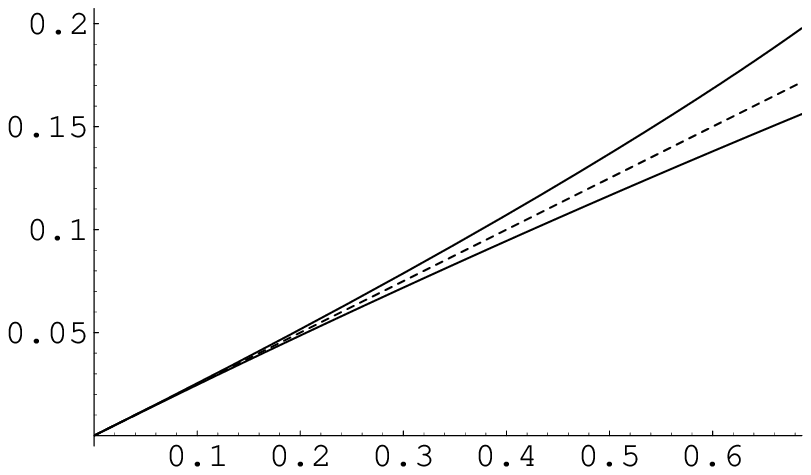}
\end{center}
\caption{Graph of $\Delta V$ versus $\hat x= {(2\pi)^2 \over {\hat L}^2}$}
\label{fig2}
\end{figure}
\begin{figure}[h]
\begin{center}
\includegraphics[scale=0.85]{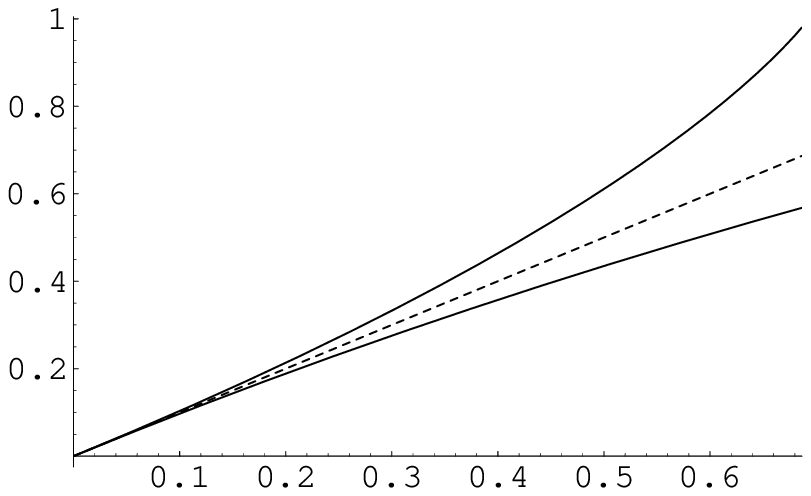}
\end{center}
\caption{Graph of $\ca$ versus $\hat x= {(2\pi)^2 \over {\hat L}^2}$}
\label{fig3}
\end{figure}

In particular, taking $\hat L = 7.5832$ in the above gives 
the following numerical estimates. 

\begin{corollary} Let $X$ be a compact, orientable $3$-manifold as in Theorem \ref{multicusp_thm},  and let $c\in H_1(\bd X;\R)$
be a surgery coefficient with $\hat L(c) > 7.5832$.
Then
\begin{enumerate}
\item the decrease in volume during hyperbolic Dehn filling is at most $0.198$,
\item the total visual area of the boundary of the filled hyperbolic manifold $M(c)$ is at most 
$h(R_0) \approx 0.980254$. 
\end{enumerate}
In particular, if $M(c)$ is a hyperbolic manifold with a smooth core geodesic
of length $\ell$ then $\ca = 2\pi \ell$.
Thus the core geodesic length is at most ${h(R_0)\over 2\pi} \approx 0.156012$.
\end{corollary}

\end{document}